\documentclass[12pt]{article}
\usepackage[a4paper,left=2cm,right=2cm,top=2cm,bottom=3cm]{geometry}  
\usepackage{amsfonts,amsmath,latexsym,amssymb}
\usepackage{amsthm}
\usepackage{mathrsfs,upref}

\usepackage[unicode]{hyperref}
\hypersetup{
            pdfauthor=Michal Bathory,
            hidelinks,          
            pdfstartview=FitH                
}

\newtheorem{theorem}{Theorem}
\numberwithin{theorem}{section}
\newtheorem{lemma}[theorem]{Lemma}
\newtheorem{proposition}[theorem]{Proposition}
\newtheorem{cor}[theorem]{Corollary}

\theoremstyle{definition}
\newtheorem{definition}[theorem]{Definition}
\newtheorem{remark}[theorem]{Remark}

\numberwithin{equation}{section}       

\newcommand{\R}{\mathbb{R}}
\newcommand{\CC}{\mathfrak{C}}
\newcommand{\N}{\mathbb{N}}

\newcommand{\f}[2]{\frac{#1}{#2}}
\newcommand{\tf}[2]{\tfrac{#1}{#2}}
\newcommand{\dd}[1]{\,\mathrm{d}{#1}}
\newcommand{\intt}[4]{\int_{#1}^{#2}\!#3\dd{#4}}
\newcommand{\eps}{\varepsilon}
\newcommand{\Meas}{\mathcal{M}}
\newcommand{\Pos}[1]{\Meas^+(#1)}
\newcommand{\upa}{\uparrow}
\newcommand{\dupa}{\downarrow}
\newcommand{\Upa}[1]{\Meas^+(#1;\upa)}
\newcommand{\Dupa}[1]{\Meas^+(#1;\dupa)}
\newcommand{\Wg}{\mathcal{W}}
\newcommand{\es}{\approx}
\newcommand{\ls}{\lesssim}
\newcommand{\gs}{\gtrsim}
\newcommand{\s}{\sigma}
\newcommand{\rh}{\varrho}
\newcommand{\lra}{\longrightarrow}

\newcommand{\con}[1]{\widetilde{#1}}
\newcommand{\embl}{\hookrightarrow}

\newcommand{\LHS}[1]{\text{LHS}\eqref{#1}}
\newcommand{\RHS}[1]{\text{RHS}\eqref{#1}}
\newcommand{\JW}{\text{\textup{JW}}}
\newcommand{\LBl}{\text{\textup{LB}}_1}
\newcommand{\LBr}{\text{\textup{LB}}_2}
\newcommand{\SV}{\text{\textup{SV}}}
\newcommand{\la}{\lambda}
\newcommand{\el}{\ell_1}
\newcommand{\eld}{\ell_2}
\newcommand{\elt}{\ell_3}
\newcommand{\g}{\Gamma}

\DeclareMathOperator*{\esssup}{ess\,sup}

\usepackage{mleftright,xparse}

\NewDocumentCommand\xDeclarePairedDelimiter{mmm}
 {%
  \NewDocumentCommand#1{som}{%
   \IfNoValueTF{##2}
    {\IfBooleanTF{##1}{#2##3#3}{\mleft#2##3\mright#3}}
    {\mathopen{##2#2}##3\mathclose{##2#3}}%
  }%
 }
\xDeclarePairedDelimiter{\abs}{\lvert}{\rvert}
\xDeclarePairedDelimiter{\set}{\lbrace}{\rbrace}
\xDeclarePairedDelimiter{\norm}{\lVert}{\rVert}

\begin{document}
\begin{center}
{\huge Joint weak type interpolation on Lorentz-Karamata spaces}\\\vskip7mm
{\Large Michal Bathory}\footnote{Mathematical Institute of the Charles University 186 75 Praha 8, Sokolovsk\'{a} 83, Czech Republic\\
e-mail: \texttt{bathory@karlin.mff.cuni.cz.}\\ The author was supported by the grant SVV-2016-260335 and the project UNCE 204014.
}
\end{center}

\begin{abstract}
We present sharp interpolation theorems, including all limiting cases, for a class of quasilinear operators of joint weak type acting between Lorentz-Karamata spaces over $\s$-finite measure. This class contains many of the important integral operators. The optimality in the scale of Lorentz-Karamata spaces is also discussed. The proofs of our results rely on a characterization of Hardy-type inequalities restricted to monotone functions and with power-slowly varying weights. Some of the limiting cases of these inequalities have not been considered in the literature so far.
\end{abstract}
\section{Introduction}
The concept of Lorentz-Karamata (LK) spaces is a natural generalization of the generalized Lorentz-Zygmund (GLZ) spaces, that has been proven to be very useful when one needs to find a precise description of the boundedness of the given operator, especially in the limiting cases (cf. \cite{BR}, \cite{EOP}, for example). It also seems that the LK spaces lead to an optimal balance between the generality and explicitness of the resulting theorems.


The main results of this paper are formulated in Section~\ref{S3} and proved in Section~\ref{S5}. Those are the interpolation theorems for quasilinear operators of joint weak type, i.e.\!\! operators, which are, in certain sense, dominated by the Calder\'{o}n operator (see \eqref{S} below). This class contain many important operators (e.g.\! convolution or singular integral operators) and thus, our results are widely applicable. We will illustrate this on several examples in Section~\ref{S7}. The assumption that some operator is of joint weak type allows to reduce the question of its boundedness to the question of the validity of certain Hardy-type inequality, restricted to non-increasing functions. Thus, the essential part of this paper is to find necessary and sufficient conditions for this kind of inequalities to hold - this is the content of Section~\ref{S4}. Moreover, since weights appearing in those inequalities are of a special (and yet very general) form ($w(x)=x^{\alpha}b(x)$, where $\alpha\in\R$ and $b$ is a slowly varying function) we are able to pinpoint the cases, where the restriction of these inequalities to monotone functions plays any role. In fact, we will show that in most of the cases it is sufficient to apply the known criteria for non-restricted weighted Hardy inequalities (Theorem~\ref{TH} below) and some rather elementary arguments. However, there are certain limiting cases in which one requires a different approach to obtain sharp results. It turns out that these problematic cases can occur only for certain subclass of considered operators; the Hilbert transform is the canonical example. Thus, the characterization of its boundedness in the limiting cases is, in a sense, the most challenging and this will be our ultimate goal. 

Our work extends the results of several papers. In \cite{BR} the authors already use the notion of joint weak type and develop an interpolation theory for operators acting between Lorentz-Zygmund spaces over $\s$-finite measure. We, on the other hand, work with the more general scale of spaces and also we clarify the connection with corresponding Hardy-type inequalities, which we characterize fully and which are certainly of independent interest. This allows us to prove also the necessity of obtained conditions. In \cite{Saw} the author gives necessary and sufficient conditions for the boundedness of several important operators acting between the classical Lorentz spaces. However, some of the limiting cases (when the Lorentz space index $r$ is $1$ or $\infty$) are missing there and the used methods does not apply to them (the case $0<r\leq 1$ was eventually described by M. Carro and J. Soria in \cite{CS}). Moreover, unlike in both articles \cite{Saw} and \cite{BR}, we discuss also the optimality (or sharpness) of obtained results (see Section~\ref{S6}). Finally, we extend the theory presented in \cite{EOP} by considering more general spaces over only $\s$-finite measure and consequently, by proving more general Hardy-type inequalities for the whole interval $(0,\infty)$.

\section{Preliminaries}\label{S2}
The following conventions are used throughout this paper: $\infty:=+\infty$, $\tf00:=0$, $\tf{c}{\infty}:=0$, $\tf{c}{0}:=\infty$, for $c\in(0,\infty]$. We also put $\infty\cdot0=0\cdot\infty:=0$. The conjugate index $p'$ to $p\in[1,\infty]$ is defined by $\tf1p+\tf1{p'}=1$. The symbol $\chi_I$ stands for the characteristic function of an interval $I\subseteq\R$. The abbreviations $\text{LHS}(\#)$ or $\text{RHS}(\#)$ are used for the left-hand side or the right-hand side of the relation $(\#)$.

For two non-negative expressions $E,F$, we shall write $E\ls F$ or equivalently $F\gs E$ if there is a~constant $c\in(0,\infty)$ such that $E\leq cF$ and $c$ is independent of appropriate quantities involved in $E$, $F$. Typically, $c$ will always be independent of functions $f,g,h$ and variables $x,t,u,\tau$, but can depend on any other symbol. When $E\ls F\ls E$, we say that $E$ is equivalent to $F$ and we will denote this by $E\es F$.

\subsection*{The decreasing rearrangement}
Let $(R,\mu)$ be a measure space with $\sigma$-finite measure $\mu$. If $\mu(R)<\infty$, we will suppose $\mu(R)=1$ without loss of generality. We denote by $\Meas(R,\mu)$ the set of all scalar valued (real or complex) $\mu$-measurable functions defined on $R$. The symbol $\Meas^+(A,B)$ stands for the set of non-negative, measurable (with respect to Lebesgue measure on $\R$) functions defined on the interval $(A,B)$, which is always one of the intervals $(0,1),(1,\infty),(0,\infty)$. Moreover, the symbols $\Dupa{A,B}$ and $\Upa{A,B}$ denote the sets of all functions from $\Pos{A,B}$ which are non-increasing and non-decreasing, respectively. By $\norm{\cdot}_{r,(A,B)}$, $1\leq r\leq\infty$, we shall denote the usual Lebesgue space norm over $(A,B)$.

The distribution function $d$ of $f$ with respect to $\mu$ is defined by
$$d(\mu,f)(h)=\mu(\set{x\in R:\abs{f(x)}>h}),\quad h\geq0.$$
The decreasing rearrangement of $f$ is then given by
$$f^*(t)=f^*_{(R,\mu)}(t)=\inf\set{h>0:d(\mu,f)(h)\leq t},\quad t\in(0,\infty).$$
We say that functions $f\in(R_1,\mu_1)$ and $g\in(R_2,\mu_2)$ are equimeasurable if their distribution functions are the same, i.e. if $d(\mu_1,f)=d(\mu_2,g)$.
See \cite[Chapter 2, Section 1]{BS} for details.

\subsection*{The Calder\'{o}n operator}
Suppose $1\leq p_1<p_2\leq\infty$, $1\leq q_1,q_2\leq\infty$, $q_1\neq q_2$. The Calder\'{o}n operator $S_\sigma$ associated with the interpolation segment $\sigma=[(\f1{p_1},\f1{q_1});(\f1{p_2},\f1{q_2})]$ is defined for every $g\in\Pos{0,\infty}$ and all $x\in(0,\infty)$ as
\begin{equation}\label{S}
S_\sigma g(x)=x^{-\f1{q_1}}\intt{0}{x^m}{t^{\f1{p_1}-1}g(t)}{t}+x^{-\f1{q_2}}\intt{x^m}{\infty}{t^{\f1{p_2}-1}g(t)}{t},
\end{equation}
where $m=(\f1{q_1}-\f1{q_2})(\f1{p_1}-\f1{p_2})^{-1}$ denotes the slope of the segment $\sigma$. 

Throughout the paper we consider only such operators $T$ which take some linear subspace $\mathcal{D}$ of $\Meas(R_1,\mu_1)$ into $\Meas(R_2,\mu_2)$. The operator $T$ is quasilinear if there is $k\geq1$ such that
$$\abs{T(f+g)}\leq k(\abs{Tf}+\abs{Tg})\quad\text{and}\quad\abs{T(\alpha f)}=\abs{\alpha}\,\abs{Tf},$$
$\mu_2\text{-a.e. on }R_2$, for every $f,g\in\mathcal{D}$ and all $\alpha\in\mathbb{C}$. Let us denote $\mathcal{D}_{S}$ the set of all functions $f\in\Meas(R_1,\mu_1)$ which satisfy $S_{\sigma}f^*(1)<\infty$. The quasilinear operator $T$ is said to be of joint weak type $(p_1,q_1;p_2,q_2)$ (notation $T\in\JW(p_1,q_1;p_2,q_2)$) if $\mathcal{D}_{S}\subseteq \mathcal{D}$ and
$$(Tf)^*(x)\ls S_\sigma f^*(x)\quad\forall x\in(0,\infty)\quad\forall f\in\mathcal{D}_{S}.$$
We write $T\in\LBl(p_1,q_1;m)$, or $T\in\LBr(p_2,q_2;m)$, for a~quasilinear operator $T$ if, for any $f\in\Dupa{0,\infty}$, there is a~function $g\in\Meas(R_1,\mu_1)$ equimeasurable with $f$ such that, for all $x\in(0,\infty)$,
$$(Tg)^*(x)\gs x^{-\f1{q_1}}\intt{0}{x^m}{t^{\f1{p_1}-1}f(t)}{t},\quad
\text{or}\quad(Tg)^*(x)\gs x^{-\f1{q_2}}\intt{x^m}{\infty}{t^{\f1{p_2}-1}f(t)}{t},$$
respectively.

If $X$ and $Y$ are two (quasi-) normed spaces, then the symbol $T:X\lra Y$ means that $T$ is bounded from $X$ to $Y$ (i.e. $\norm{Tf}_Y\ls \norm{f}_X$ for all $f\in X$). Furthermore, the symbol $X\embl Y$ stands for $id:X\lra Y$.

\subsection*{Slowly varying functions}
The function $a\in\Pos{A,B}$, $0\not\equiv a\not\equiv\infty$, is said to be slowly varying (s.v.) on $(A,B)$ if, for each $\eps>0$, there exist functions $g_{\eps}\in\Upa{A,B}$, $g_{-\eps}\in\Dupa{A,B}$ such that
\begin{equation}\label{eq8}
t^{\eps}a(t)\es g_{\eps}(t)
\quad\text{and}\quad
t^{-\eps}a(t)\es g_{-\eps}(t)
\quad\forall t\in(A,B).
\end{equation}
We denote by $\SV(A,B)$ the set of all slowly varying functions on $(A,B)$. 

We shall now review some important properties of the slowly varying functions. The most basic ones contained in the following proposition are used in the paper without reference.

\begin{proposition}\label{Prop}
Let $a,b\in \SV(A,B)$.
\begin{itemize}
\item[{\rm(i)}] All of the functions $ab$, $\tf1a$, $a^r$, $t\mapsto a(t^r)$, $r\geq0$, are slowly varying on $(A,B)$.
\item[{\rm(ii)}] Let $[C,D]\subseteq[A,B]\cap(0,\infty)$. Then there exist constants $c_1,c_2\in(0,\infty)$ such that $c_1\leq a(x)\leq c_2$ for all $x\in[C,D]$.
\item[{\rm(iii)}] If $c>0$, then $a(ct)\es a(t)$ for every $t\in(0,\infty)$.
\end{itemize}
\end{proposition}
\begin{proof}
For (i) and (iii), see \cite[Proposition 2.2 (i), (ii), (iii)]{GOT}.

Clearly, it is sufficient to prove the assertion (ii) in the case $(A,B)=(0,\infty)$. Let $\eps>0$. By \eqref{eq8}, there exists a function $g_{-\eps}\in\Dupa{0,\infty}$ which is equivalent to the function $x\mapsto x^{-\eps}a(x)$. Then, for all $x\in[C,D]$,
$$a(x)=x^{\eps}x^{-\eps}a(x)\es x^{\eps}g_{-\eps}(x)\leq D^{\eps}g_{-\eps}(C)=:c_2.$$
The existence of the lower bound then follows from (i) ($\tf1a$ is also slowly varying).
\end{proof}

\begin{lemma}\label{LSV}
Let $\la\in\SV(0,\infty)$ and $r\in[1,\infty]$.
\begin{itemize}
\item[\rm{(i)}] If $\eps>0$, then 
$$\norm{t^{\eps-\f1r}\la(t)}_{r,(0,x)}\es x^{\eps}\la(x)\quad\text{and}\quad\norm{t^{-\eps-\f1r}\la(t)}_{r,(x,\infty)}\es x^{-\eps}\la(x)\quad\forall x\in(0,\infty).$$
\item[\rm{(ii)}] Then
$$\norm{t^{-\f1r}\la(t)}_{r,(0,x)}\gs\la(x)\quad\text{and}\quad\norm{t^{-\f1r}\la(t)}_{r,(x,\infty)}\gs\la(x)\quad\forall x\in(0,\infty).$$
Furthermore, if $\norm{t^{-\f1r}\lambda(t)}_{r,(0,1)}\!\!\!<\infty$ and $\norm{t^{-\f1r}\lambda(t)}_{r,(1,\infty)}\!\!\!<\infty$, then the functions $x\mapsto\norm{t^{-\f1r}\la(t)}_{r,(0,x)}$ and $x\mapsto\norm{t^{-\f1r}\la(t)}_{r,(x,\infty)}$ belong to $\SV(0,\infty)$, respectively.
\end{itemize}
\end{lemma}
\begin{proof}
For (i) see \cite[Proposition 2.2 (iv)]{GOT}. An important consequence of (i) is that every slowly varying function is equivalent to some continuous function. This fact implies (ii) in the case $r=\infty$. When $r<\infty$, we can write
$$\norm{t^{-\f1r}\lambda(t)}_{r,(0,x)}
=\left(\intt{0}{x}{t^{-1}\lambda(t)^r}{t}\right)^{\f1r}
\gs\left(x^{-1}\lambda(x)^r\intt{0}{x}{1}{t}\right)^{\f1r}
=\lambda(x)$$
and
$$\norm{t^{-\f1r}\lambda(t)}_{r,(x,\infty)}
=\left(\intt{x}{\infty}{t^{-1-\eps}\;t^{\eps}\lambda(t)^r}{t}\right)^{\f1r}
\gs\left(x^{\eps}\lambda(x)^r\intt{x}{\infty}{t^{-1-\eps}}{t}\right)^{\f1r}
\es\lambda(x)$$
for all $x\in(0,\infty)$. For the last assertion of (ii) see \cite[Lemma 2.1. (v)]{GO}.
\end{proof}

\begin{lemma}\label{L}
Let $R\in[1,\infty)$, $S\in[1,\infty]$, $\la\in\SV(A,B)$ and set
$$\Lambda_1(x)=\intt{A}{x}{t^{-1}\lambda(t)^R}{t}\quad\text{and}\quad\Lambda_2(x)=\intt{x}{B}{t^{-1}\lambda(t)^R}{t},\quad x\in(A,B).$$
\begin{itemize}
\item[{\rm (i)}] Suppose that
\begin{equation}\label{1001}
\intt{A}{B}{t^{-1}\lambda(t)^R}{t}=\infty.
\end{equation}
Then
\begin{align}\label{eq34}
\norm{t^{-\f1R}\la(t)}_{R,(A,x)}
&\es\norm{t^{-\f1S}\la(t)^{\f{R}{S}}\Lambda_1(t)^{-\f1R-\f1S}}_{S,(x,B)}^{-1}
\end{align}
and
\begin{align}\label{eq35}
\norm{t^{-\f1R}\la(t)}_{R,(x,B)}
&\es\norm{t^{-\f1S}\la(t)^{\f{R}{S}}\Lambda_2(t)^{-\f1R-\f1S}}_{S,(A,x)}^{-1}
\end{align}
for all $x\in(A,B)$.
\item[{\rm(ii)}] Suppose $\delta\in(0,1)=(A,B)$. Then \eqref{eq34} holds for all $x\in(0,\delta)$.
\item[{\rm(iii)}] Suppose $\delta\in(1,\infty)=(A,B)$. Then \eqref{eq35} holds for all $x\in(\delta,\infty)$.
\end{itemize}
\end{lemma}
\begin{proof}
\mbox{}

\textit{Case} (i). We prove relation \eqref{eq34} here, the proof of \eqref{eq35} is analogous.

If $S=\infty$, then \eqref{eq34} is in fact an equality. It can also happen that both sides of \eqref{eq34} are identically infinite. In other cases, we use the change of variables $\tau=\Lambda_1(t)$ and \eqref{1001} to get, for all $x\in(A,B)$, that
$$\RHS{eq34}=\left(\intt{x}{B}{t^{-1}\la(t)^{R}\Lambda_1(t)^{-\f SR-1}}{t}\right)^{-\f1S}
=\left(\intt{\Lambda_1(x)}{\infty}{\tau^{-\f{S}{R}-1}}{\tau}\right)^{-\f1S}
\es\LHS{eq34}.$$

\textit{Case} (ii). We proceed in the same way as in (i) to get
$$\RHS{eq34}\es(\Lambda_1(x)^{-\f SR}-\Lambda_1(1)^{-\f SR})^{-\f1S}\geq\LHS{eq34}$$
for all $x\in(0,1)$. Moreover, since the function $t\mapsto \Lambda_1(t)^{-\f SR}$ is strictly decreasing on $(0,1)$, it follows that
$$(\Lambda_1(x)^{-\f SR}-\Lambda_1(1)^{-\f SR})^{-\f1S}\leq (c(\delta)\Lambda_1(x)^{-\f SR})^{-\f1S}\es\LHS{eq34}$$
for all $x\in(0,\delta)$, where $c(\delta)=1-\Lambda_1(\delta)^{\f SR}\Lambda_1(1)^{-\f SR}>0$.

\textit{Case} (iii) can be proven analogously as case (ii).
\end{proof}

\subsection*{The Lorentz-Karamata spaces}
\begin{definition}
Let $0<p,r\leq\infty$, $a\in\SV(A,B)$ and put
$$\norm{f}_{p,r;a;(A,B)}:=\norm{t^{\f1p-\f1r}a(t)f^*(t)}_{r,(A,B)},\quad f\in\Meas(R,\mu).$$
Let $B=1$ if $\mu(R)=1$ and $B=\infty$ if $\mu(R)=\infty$. Then, the Lorentz-Karamata (LK) space $L_{p,r;a}(R,\mu)\equiv L_{p,r;a}$ is defined as the set of all functions $f\in\Meas(R,\mu)$ such that $\norm{f}_{p,r;a;(0,B)}<\infty$.
\end{definition}
Using the monotonicity of $f^*$ and Lemma~\ref{LSV}~(i), one can observe that $L_{p,r;a}$ is the trivial space if and only if $p=\infty$ and $\norm{t^{-\f1r}a(t)}_{r,(0,1)}=\infty$. Moreover, if we set $d(t)=\norm{a}_{\infty,(0,t)}$, $t\in(0,B)$, and if the space $L_{\infty,\infty;a}$ is non-trivial, then $d\in\SV(0,B)$ (see Lemma~\ref{LSV}~(ii)) and
$$\norm{af^*}_{\infty,(0,B)}
\ls\norm{\norm{a}_{\infty,(0,t)}f^*(t)}_{\infty,(0,B)}
\leq\norm{\norm{af^*}_{\infty,(0,t)}}_{\infty,(0,B)}
=\norm{af^*}_{\infty,(0,B)}.$$
Thus, $L_{\infty,\infty;a}=L_{\infty,\infty;d}$ and, consequently, in the case $p=\infty$ it is natural to assume that
\begin{equation}\label{210}
\text{if}\quad r=\infty,\quad\text{then}\quad a\in\Upa{0,B}.
\end{equation}

LK spaces contain many of familiar spaces as particular cases. For example, let $\el(t)=1+\abs{\log t}$, $t\in(0,\infty)$, and $\ell_i=\el(\ell_{i-1})$ for all $i\in\set{2,3,\ldots}$ and set $\mathcal{L}=\prod_{i=1}^{n}\ell_i^{\alpha_i}$, where $\alpha_i\in\R$, $n\in\N$.
Then $\mathcal{L}\in\SV(0,\infty)$ and $L_{p,r;\mathcal{L}}$ is the~generalized Lorentz-Zygmund (GLZ) space with the $n$-th tier of logarithm. In particular, if $\alpha,\beta,\gamma\in\R$, then $L_{p,r;\alpha,\beta,\gamma}:=L_{p,r;\el^{\alpha}\eld^{\beta}\elt^{\gamma}}$ and $L_{p,r;\alpha,\beta}:=L_{p,r;\el^{\alpha}\eld^{\beta}}$ are the GLZ spaces of Edmunds, Gurka and Opic (cf. \cite{EGO}) and $L_{p,r;\el^{\alpha}}$ is the Lorentz-Zygmund space of Bennett and Rudnick (cf. \cite{BS}). The LK spaces also cover the (generalized) Lorentz-Zygmund spaces $L_{p,r;\mathbb{A}}$, $\mathbb{A}=(\alpha_0,\alpha_{\infty})\in\R^2$, with ``broken-logartmic'' function, which were introduced in \cite{EOP2}. Furthermore, if $\mathfrak{1}=\chi_{(0,\infty)}$, then $L_{p,r}:=L_{p,r;\mathfrak{1}}$ is the Lorentz space. Moreover, the space $L_p(\log{L})^{\alpha}:=L_{p,p;\el^{\alpha}}$ is the Zygmund space, and $L_p:=L_{p,p}$ is the~Lebesgue space (original definitions and properties of these classical spaces can be found also in \cite{BS}). In the literature also spaces, which are close to $L_{\infty}$, such as $L_{\exp}^{\alpha}$, appear. These spaces are covered by the LK spaces as well ($L^{\alpha}_{\exp}=L_{\infty,\infty;\el^{-\alpha}}$). Since the special spaces mentioned above were introduced by different authors at various times, there is slight inconsistency in their definitions (many functionals can be used to define the same space). This is resolved by \cite[Lemma~2.2.]{EOP} (the definitions of the GLZ space from \cite{EOP} and of the LK space given here are consistent). 

The choice of slowly varying function $a$ is, of course, not restricted to composite logarithmic functions as $\mathcal{L}$. For complete information on how various examples of slowly varying functions can be constructed, see \mbox{\cite[Section 1.3, p. 12]{Reg}}. Note that a~general slowly varying function can also exhibit oscillations of infinite amplitude at zero. An example of such slowly varying function is $a(x)=\exp\left(\el(x)^{\f13}\cos(\el(x)^{\f13})\right)$, $x\in(0,\infty)$, which is taken from \cite[Exercise 1.11.3, p. 58]{Reg}.

Similarly as in \cite{BR}, we shall also consider sums and intersections of the LK spaces.
\begin{definition}\label{DSum}
Let $p_1,p_2,r_1,r_2\in[1,\infty]$, $p_1\neq p_2$, $a\in\SV(0,\infty)$, $f\in\Meas(R,\mu)$. Then, we define
$$\norm{f}_{(p_1,r_1)+(p_2,r_2);a}:=\left\{
\begin{array}{cc}
\norm{f}_{p_1,r_1;a;(0,1)}+\norm{f}_{p_2,r_2;a;(1,\infty)} & \text{if  } p_1<p_2\\
\norm{f}_{p_2,r_2;a;(0,1)}+\norm{f}_{p_1,r_1;a;(1,\infty)} & \text{if  } p_1>p_2
\end{array}\right.$$
and
$$\norm{f}_{(p_1,r_1)\cap(p_2,r_2);a}:=\norm{f}_{(p_2,r_2)+(p_1,r_1);a}.$$
The spaces $L_{p_1,r_1;a}+L_{p_2,r_2;a}$ and $L_{p_1,r_1;a}\cap L_{p_2,r_2;a}$ consist of all functions $f\in\Meas(R,\mu)$, such that $\norm{f}_{(p_1,r_1)+(p_2,r_2);a}<\infty$ and $\norm{f}_{(p_1,r_1)\cap(p_2,r_2);a}<\infty$, respectively.
\end{definition}

Obviously, this definition enables us to control the behaviour of $f^*$ near $0$ and $\infty$ independently. Also, it agrees with the usual definition of sum and intersection of spaces (therefore the notation) up to one exception, when one of the spaces of the sum is trivial. This exception allows us to properly define, e.g., the space \mbox{$L(\log L)+L_{\infty,1}$}, which is particularly important for Hilbert transform and which, using the usual definition of the sum, would be trivial (since $L_{\infty,1}=\set{0}$). For detailed explanation of this problematic, see \cite{BR}.

\section{The statement of the main results}\label{S3}
If not stated otherwise, we shall assume in this section that $1\leq p_1<p_2\leq\infty$, $1\leq q_1,q_2\leq\infty$, $q_1\neq q_2$, $1\leq r,s,r_1,s_1,r_2,s_2\leq\infty$ and $a,b\in\SV(A,B)$. If $r>s$, then the number $\rh$ is defined by
\begin{equation*}
\f1{\rh}=\f1s-\f1r.
\end{equation*}
To formulate our main results conveniently, we introduce the following quantities. Set
$$N(r,s,a,b;A,B)=\left\{\begin{array}{cc}
\sup\limits_{A<x<B}b(x)a(x)^{-1} & \text{if }r\leq s\\
\norm{x^{-\f1{\rh}}b(x)a(x)^{-1}}_{\rh,(A,B)} & \text{if }r>s
\end{array}\right.,$$

$$L(r,s,a,b;A,B)=\left\{\begin{array}{cc}
\sup\limits_{A<x<B}\norm{t^{-\f1s}b(t)}_{s,(x,B)}\norm{t^{-\f1{r'}}a(t)^{-1}}_{r',(A,x)} & \text{if }r\leq s\\
\norm{x^{-\f1{\rh}}a(x)^{-\f{r'}{\rh}}\norm{t^{-\f1s}b(t)}_{s,(x,B)}\norm{t^{-\f1{r'}}a(t)^{-1}}_{r',(A,x)}^{\f{r'}{s'}}}_{\rh,(A,B)} & \text{if }r>s
\end{array}\right.,$$
and
$$R(r,s,a,b;A,B)=\left\{\begin{array}{cc}
\sup\limits_{A<x<B}\norm{t^{-\f1s}b(t)}_{s,(A,x)}\norm{t^{-\f1{r'}}a(t)^{-1}}_{r',(x,B)} & \text{if }r\leq s\\
\norm{x^{-\f1{\rh}}a(x)^{-\f{r'}{\rh}}\norm{t^{-\f1s}b(t)}_{s,(A,x)}\norm{t^{-\f1{r'}}a(t)^{-1}}_{r',(x,B)}^{\f{r'}{s'}}}_{\rh,(A,B)} & \text{if }r>s
\end{array}\right..$$
Furthermore, we put
\begin{align*}
R_1(r,s&,a,b;A,B)\!=\!\left\{\begin{array}{cc}
\sup\limits_{A<x<B}\norm{t^{-\f1s}b(t)\log\tf xt}_{s,(A,x)}\norm{t^{-\f1{r}}a(t)}^{-1}_{r,(A,x)} & \text{if }r\leq s\\
\norm{x^{-\f1{\rh}}a(x)^{\f r{\rh}}\norm{t^{-\f1s}b(t)\log \tf xt}_{s,(A,x)}\norm{t^{-\f1{r}}a(t)}_{r,(A,x)}^{-\f rs}}_{\rh,(A,B)} & \text{if }r>s
\end{array}\right.,\\
R_2(r,s&,a,b;A,B)\\
&=\left\{\begin{array}{cc}
\sup\limits_{A<x<B}\norm{t^{-\f1s}b(t)}_{s,(A,x)}\norm{t^{-\f1{r'}}a(t)^{\f r{r'}}V(t)^{-1}\log\tf tx}_{r',(x,B)} & \text{if }r\leq s\\
\norm{x^{-\f1{\rh}}b(x)^{\f{s}{\rh}}\norm{t^{-\f1s}b(t)}_{s,(A,x)}^{\f sr}\norm{t^{-\f1{r'}}a(t)^{\f r{r'}}V(t)^{-1}\log\tf tx}_{r',(x,B)}}_{\rh,(A,B)} & \text{if }r>s
\end{array}\right.,
\end{align*}
where $V(t)=\intt{A}{t}{u^{-1}a(u)^r}{u}$, $t\in(A,B)$, and
$$R_3(r,s,a,b;A,B)=\norm{x^{-\f1s}b(t)\intt{x}{B}{t^{-1}\norm{a}_{\infty,(A,t)}^{-1}}{t}}_{s,(A,B)}\quad\text{if }r=\infty.$$
Finally, let
$$R_{\infty}(r,s,a,b;A,B)=\left\{\begin{array}{cl}
R_1(r,s,a,b;A,B)+R_2(r,s,a,b;A,B) & \text{if } 1<r,s<\infty\\
R_1(r,s,a,b;A,B) & \text{if }r=1 \text{ or } 1=s<r<\infty\\
R_2(r,s,a,b;A,B) & \text{if }1<r<s=\infty\\ 
R_3(r,s,a,b;A,B) & \text{if }r=\infty.
\end{array}\right.$$

Whenever the context is clear, we shall write just $N$ instead of $N(r,s,a,b;A,B)$ and similarly for all the other quantities above.

Now we are almost ready to formulate our interpolation theorems. We recall that we work with the operators acting between (subspaces of) $\Meas(R_1,\mu_1)$ and $\Meas(R_2,\mu_2)$. We shall suppose that $\mu_1(R_1)=\mu_2(R_2)=\infty$; for the finite measure spaces see Remark~\ref{Rem}~(i) below. If $b\in\SV(A,B)$, we put
$$b_*(t)=b(t^{\f1m}),\quad t\in(A,B),$$
where $m$ denotes the slope of the interpolation segment $\sigma=[(\f1{p_1},\f1{q_1});(\f1{p_2},\f1{q_2})]$, associated with the Calder\'{o}n operator $S_{\sigma}$.

The following theorem is a~generalization of the classical Marcinkiewicz interpolation theorem (cf. \cite[Chapter 4, Theorem 4.13]{BS}) to the LK spaces. 

\begin{theorem}\label{Int}
Let $T\in\JW(p_1,q_1;p_2,q_2)\cap(\LBl(p_1,q_1;m)\cup\LBr(p_2,q_2;m))$. Suppose that $\theta\in(0,1)$ and $p,q$ satisfy
$$\f1p=\f{1-\theta}{p_1}+\f{\theta}{p_2},\quad\f1q=\f{1-\theta}{q_1}+\f{\theta}{q_2}.$$
Then 
\begin{equation*}
T:L_{p,r;a}\lra L_{q,s;b}
\end{equation*}
if and only if
$$N(r,s,a,b_*;0,\infty)<\infty.$$
\end{theorem}

The parameter $\theta$ from the previous theorem was restricted to $(0,1)$, therefore we refer to this case as to the non-limiting case. The next theorem describes the limiting case $\theta=0$.

\begin{theorem}\label{IntL}
Let $T\in\JW(p_1,q_1;p_2,q_2)\cap\LBl(p_1,q_1;m)$. Then 
\begin{equation*}
T:L_{p_1,r;a}\lra L_{q_1,s;b}
\end{equation*}
if and only if
$$L(r,s,a,b_*;0,\infty)<\infty.$$
\end{theorem}

The following theorem describes the limiting case $\theta=1$ and is completely analogical to the previous theorem as long as $p_2<\infty$.

\begin{theorem}\label{IntR}
Let $T\in\JW(p_1,q_1;p_2,q_2)\cap\LBr(p_2,q_2;m)$, $p_2<\infty$. Then 
\begin{equation*}
T:L_{p_2,r;a}\lra L_{q_2,s;b}
\end{equation*}
if and only if
$$R(r,s,a,b_*;0,\infty)<\infty.$$
\end{theorem}

When $p_2=\infty$, the situation turns out to be more delicate.

\begin{theorem}\label{IntRR}
Let $T\in\JW(p_1,q_1;\infty,q_2)\cap\LBr(\infty,q_2;m)$ and suppose that \eqref{210} is satisfied. Then
$$T:L_{\infty,r;a}\lra L_{q_2,s;b}$$
if and only if
\begin{equation}\label{1234}
\norm{t^{-\f1r}a(t)}_{r,(0,\infty)}=\infty
\end{equation}
and 
$$R_{\infty}(r,s,a,b_*;0,\infty)<\infty.$$
\end{theorem}

Next, we state results concerning the sums and intersections of LK spaces. We concentrate on the limiting cases only as the situation in the non-limiting case is obvious.

\begin{theorem}\label{IntLim}
Let $T\in\JW(p_1,q_1;p_2,q_2)\cap\LBl(p_1,q_1;m)\cap\LBr(p_2,q_2;m)$ with $p_2<\infty$.
\begin{itemize}
\item[{\rm(i)}] Then
\begin{equation*}
T:L_{p_1,r_1;a}+ L_{p_2,r_2;a}\lra L_{q_1,s_1;b}+L_{q_2,s_2;b}
\end{equation*}
if and only if 
$$L(r_1,s_1,a,b_*;0,1)+R(r_2,s_2,a,b_*;1,\infty)<\infty.$$
\item[{\rm(ii)}] Then
\begin{equation*}
T:L_{p_1,r_1;a}\cap L_{p_2,r_2;a}\lra L_{q_1,s_1;b}\cap L_{q_2,s_2;b}
\end{equation*}
if and only if 
$$L(r_1,s_1,a,b_*;1,\infty)+R(r_2,s_2,a,b_*;0,1)<\infty.$$
\end{itemize}
\end{theorem}

\begin{theorem}\label{IntHil}
Let $T\in\JW(p_1,q_1;\infty,q_2)\cap\LBl(p_1,q_1;m)\cap\LBr(\infty,q_2;m)$ and suppose that \eqref{210} is satisfied.
\begin{itemize}
\item[{\rm(i)}] Then
\begin{equation*}
T:L_{p_1,r_1;a}+ L_{\infty,r_2;a}\lra L_{q_1,s_1;b}+L_{q_2,s_2;b}
\end{equation*}
if and only if
$$\norm{t^{-\f1{r_2}}a(t)}_{r_2,(1,\infty)}=\infty$$
and 
$$L(r_1,s_1,a,b_*;0,1)+R_{\infty}(r_2,s_2,a,b_*;1,\infty)<\infty.$$
\item[{\rm(ii)}] Then
\begin{equation*}
T:L_{p_1,r_1;a}\cap L_{\infty,r_2;a}\lra L_{q_1,s_1;b}\cap L_{q_2,s_2;b}
\end{equation*}
if and only if 
$$L(r_1,s_1,a,b_*;1,\infty)+R_{\infty}(r_2,s_2,a,b_*;0,1)<\infty.$$
\end{itemize}
\end{theorem}

Let us now make some remarks about the theorems above.
\begin{remark}\label{Rem} \hspace{1pt}
(i) When the underlying measure spaces are finite and $q_1<q_2$, then Theorems~\ref{Int}, \ref{IntL}, \ref{IntR} continue to hold, provided we replace the interval $(0,\infty)$ by $(0,1)$. The same is true for Theorem~\ref{IntRR}, provided that condition \eqref{1234} is dropped. This is a consequence of the fact that the Hardy-type inequalities which will be used to prove the mentioned theorems hold on $(0,1)$ and $(0,\infty)$ in the same form (cf. Lemmas~\ref{TLH} -- \ref{LSaw} below).

(ii) It will be apparent from the proofs in Section~\ref{S5} that the existence of the lower bounds for the operator $T$ is only used to prove the necessity of the corresponding conditions. In other words, if we omit the assumptions $T\in\LBl(p_1,q_1;m)$ and $T\in\LBr(p_2,q_2;m)$, the theorems above still provide sufficient conditions for the boundedness of $T$.

(iii) There exist $r,s\in[1,\infty]$ and $a,b\in\SV(A,B)$ such that
$$R_1(r,s,a,b;A,B)+R_2(r,s,a,b;A,B)<\infty\quad\text{and}\quad R(r,s,a,b;A,B)=\infty.$$
Indeed, let $1<r<s<\infty$ and
$$a(t)=\el(t)^{-\f1r}\eld(t)^{\theta}\quad\text{and}\quad b(t)=\el(t)^{-1-\f1s}\eld(t)^{\theta+\gamma},\quad t\in(0,1),$$
where $\theta<-\f1r-\f1s$ and $0<\gamma<\f1r-\f1s$ (we prove the statement for $(A,B)=(0,1)$; other cases are similar). Then, using the substitution $\tau=\eld(t)$, we get
\begin{align*}
\norm{t^{-\f1s}b_*(t)\log\tf xt}_{s,(0,x)}&\norm{t^{-\f1r}a(t)}_{r,(0,x)}^{-1}\\
&\ls\norm{t^{-\f1s}\el(t)^{-\f1s}\eld(t)^{\theta+\gamma}}_{s,(0,x)}\norm{t^{-\f1r}\el(t)^{-\f1r}\eld(t)^{\theta}}_{r,(0,x)}^{-1}\\
&=\left(\intt{\eld(x)}{\infty}{\tau^{\theta s+\gamma s}}{\tau}\right)^{\f1s}\left(\intt{\eld(x)}{\infty}{\tau^{\theta r}}{\tau}\right)^{-\f1r}\\
&\es\eld(x)^{\theta+\gamma+\f1s}\eld(x)^{-\theta-\f1r}
=\eld(x)^{\gamma+\f1s-\f1r}\ls1
\end{align*}
for all $x\in(0,1)$ and thus $R_1(r,s,a,b;0,1)<\infty$. Now observe that in our case
$$A(t)=\intt{0}{t}{u^{-1}a(u)^r}{u}\es\eld(t)^{\theta r+1}\quad\forall t\in(0,1).$$
Consequently, using the substitutions $u=\el(t)$ and $\tau=\eld(t)$, we obtain
\begin{align}\label{190}
&\norm{t^{-\f1s}b_*(t)}_{s,(0,x)}\norm{t^{-\f1{r'}}a(t)^{\f r{r'}}A(t)^{-1}\log\tf tx}_{r',(x,1)}\nonumber\\
&\qquad\ls\norm{t^{-\f1s}\el(t)^{-1-\f1s}\eld(t)^{\theta+\gamma}}_{s,(0,x)}\norm{t^{-\f1{r'}}\el(t)^{-\f1{r'}}\eld(t)^{\theta\f r{r'}-\theta r-1}}_{r',(x,1)}\el(x)\nonumber\\
&\qquad=\left(\intt{\el(x)}{\infty}{u^{-s-1}\el(u)^{\theta s+\gamma s}}{u}\right)^{\f1s}\left(\intt{1}{\eld(x)}{\tau^{-r'(\theta+1)}}{\tau}\right)^{\f1{r'}}\el(x)\nonumber\\
&\qquad\es\el(x)^{-1}\eld(x)^{\theta+\gamma}\eld(x)^{-\theta-1+\f1{r'}}\el(x)=\eld(x)^{\gamma-\f1r}\ls1
\end{align}
for all $x\in(0,1)$ and thus, $R_2(r,s,a,b;0,1)<\infty$ holds as well. It remains to show that $R(r,s,a,b;0,1)=\infty$. We can see from \eqref{190} that
$$\norm{t^{-\f1s}b_*(t)}_{s,(0,x)}\es\el(x)^{-1}\eld(x)^{\theta+\gamma}\quad\forall x\in(0,1).$$
This, together with
\begin{align*}
\norm{t^{-\f1{r'}}a(t)^{-1}}_{r',(x,1)}
&=\norm{t^{-\f1{r'}}\el(t)^{\f1r}\eld(t)^{-\theta}}_{r',(x,1)}
=\left(\intt{1}{\el(x)}{u^{r'-1}\el(u)^{-\theta r'}}{u}\right)^{\f1{r'}}\\
&\es\el(x)\eld(x)^{-\theta}
\end{align*}
for all $x\in(0,1)$, gives
$$\norm{t^{-\f1s}b_*(t)}_{s,(0,x)}\norm{t^{-\f1{r'}}a(t)^{-1}}_{r',(x,1)}\es\eld(x)^{\gamma},$$
which tends to infinity as $x\to0_+$. When $r=1$ or $s=\infty$, the given example (with the usual modifications) works as well.
\end{remark}

\section{Weighted inequalities for integral operators}\label{S4}
We will show in Section~\ref{S5} that the boundedness of $T$ is fully determined by the validity of certain Hardy-type inequalities that are restricted to non-increasing functions. The aim of this section is to characterize weights for which these inequalities hold. By weights, we mean functions from $\Meas^+(A,B)$ that are positive and finite almost everywhere on $(A,B)$. We shall denote the set of all weights by~$\Wg(A,B)$.

First of all, we are going to state Lemmas~\ref{TLH}, \ref{LRed}, \ref{LRedR} and \ref{LSaw}, which form an essential part of the paper (they are applied to prove the main results). After that, we review some general criteria and use them to prove the lemmas. The following assertion will be used to prove Theorem~\ref{Int} (the non-limiting case).
\begin{lemma}\label{TLH}
Let $r,s\in[1,\infty]$, $a,b\in\SV(A,B)$. Suppose $\kappa\in\R$ and $\nu>\mu>0$. Then the following five conditions are equivalent: 
\begin{itemize}
\item[{\rm(i)}] $N(r,s,a,b;A,B)<\infty;$

\item[{\rm(ii)}]
$\displaystyle\norm{t^{-\mu-\f1s}b(t)\intt{A}{t}{u^{\kappa-1}g(u)}{u}}_{s,(A,B)}
\!\!\!\!\!\ls\norm{t^{-\mu+\kappa-\f1r}a(t)g(t)}_{r,(A,B)}\quad\forall g\in\Meas^+(A,B);$

\item[{\rm(iii)}]
$\displaystyle\norm{t^{-\mu-\f1s}b(t)\intt{A}{t}{u^{\nu-1}f(u)}{u}}_{s,(A,B)}
\!\!\!\!\!\ls\norm{t^{-\mu+\nu-\f1r}a(t)f(t)}_{r,(A,B)}\;\forall f\!\in\!\Dupa{A,B};$

\item[{\rm(iv)}]
$\displaystyle\norm{t^{\mu-\f1s}b(t)\intt{t}{B}{u^{\kappa-1}g(u)}{u}}_{s,(A,B)}
\!\!\!\ls\norm{t^{\mu+\kappa-\f1r}a(t)g(t)}_{r,(A,B)}\quad\forall g\in\Pos{A,B};$

\item[{\rm(v)}]
$\displaystyle\norm{t^{\mu-\f1s}b(t)\intt{t}{B}{u^{\nu-1}f(u)}{u}}_{s,(A,B)}
\!\!\!\ls\norm{t^{\mu+\nu-\f1r}a(t)f(t)}_{r,(A,B)}\quad\forall f\!\in\!\Dupa{A,B}.$
\end{itemize}
\end{lemma}

The parameter $\kappa$ can be, of course, eliminated by a suitable substitution; we keep it there just to emphasise that the inequalities above share the same structure. The following two lemmas describe the case, where $\mu$ from Lemma~\ref{TLH} is zero. They will be used to prove Theorems~\ref{IntL}, \ref{IntR} and \ref{IntLim} (limiting cases $\theta=0$, $\theta=1$ with $p_2<\infty$).

\begin{lemma}\label{LRed}
Let $r,s\in[1,\infty]$, $a,b\in\SV(A,B)$, $\kappa\in\R$ and $\nu>0$. Then the following three conditions are equivalent:
\begin{itemize}
\item[\rm{(i)}] $L(r,s,a,b;A,B)<\infty;$
\item[\rm{(ii)}] $\displaystyle\norm{t^{-\f1s}b(t)\intt{A}{t}{u^{\kappa-1}g(u)}{u}}_{s,(A,B)}
\ls\norm{t^{\kappa-\f1r}a(t)
g(t)}_{r,(A,B)}\quad\forall g\in\Pos{A,B};$
\item[\rm{(iii)}] 
$\displaystyle\norm{t^{-\f1s}b(t)\intt{A}{t}{u^{\nu-1}f(u)}{u}}_{s,(A,B)}
\ls\norm{t^{\nu-\f1r}a(t)f(t)}_{r,(A,B)}\quad\forall f\in\Dupa{A,B}.$
\end{itemize}
\end{lemma}

\begin{lemma}\label{LRedR}
Let $r,s\in[1,\infty]$, $a,b\in\SV(A,B)$, $\kappa\in\R$ and $\nu>0$. Then the following three conditions are equivalent:
\begin{itemize}
\item[\rm{(i)}] $R(r,s,a,b;A,B)<\infty;$
\item[\rm{(ii)}] $\displaystyle\norm{t^{-\f1s}b(t)\intt{t}{B}{u^{\kappa-1}g(u)}{u}}_{s,(A,B)}
\ls\norm{t^{\kappa-\f1r}a(t)
g(t)}_{r,(A,B)}\quad\forall g\in\Pos{A,B};$
\item[\rm{(iii)}] $\displaystyle\norm{t^{-\f1s}b(t)\intt{t}{B}{u^{\nu-1}f(u)}{u}}_{s,(A,B)}
\ls\norm{t^{\nu-\f1r}a(t)
f(t)}_{r,(A,B)}\quad\forall f\in\Dupa{A,B};$
\end{itemize}
\end{lemma}

Finally, in the following lemma we consider the remaining and most interesting case, that occurs when $\mu=\nu=0$ and the inequality is restricted to non-increasing functions. It will be used to prove Theorems~\ref{IntRR} and \ref{IntHil} (limiting case $\theta=1$ with $p_2=\infty$).
\begin{lemma}\label{LSaw}
Let $r,s\in[1,\infty]$ and $a,b\in\SV(A,B)$. Then
\begin{equation}\label{36}
\norm{t^{-\f1s}b(t)\intt{t}{B}{u^{-1}f(u)}{u}}_{s,(A,B)}
\ls\norm{t^{-\f1r}a(t)f(t)}_{r,(A,B)}
\end{equation}
holds for every $f\in\Dupa{A,B}$ if and only if
$$R_{\infty}(r,s,a,b;A,B)<\infty$$
and
\begin{equation}\label{A}
\norm{t^{-\f1r}a(t)}_{r,(A,B)}=\infty\quad\text{when}\quad B=\infty.
\end{equation}
\end{lemma}

Note that, due to the monotonicity of $f$, the analogy of inequality \eqref{36} for $\int_A^t$ is non-trivial only if $(A,B)=(1,\infty)$ and then it can be converted to an inequality of the same form as \eqref{36} on $(0,1)$, but restricted to non-decreasing functions. Since we will have no use for such an inequality and since the resulting characterization is not as interesting, we shall omit it.

To prove the first three of the four lemmas above, we will use the following well known characterization of weighted Hardy inequalities, for which we refer to \cite[Theorems~5.9, 5.10, 6.2, 6.3, Remark~5.5]{OK} or to \cite{Sin}.

\begin{theorem}\label{TH}
Let $v,w\in \Wg(A,B)$, $r,s\in[1,\infty]$ and let $\f1{\rh}=\f1s-\f1r$.
\begin{itemize}
\item[{\rm(i)}] Then
$$\norm{w(t)\int_{A}^{t}g}_{s,(A,B)}\ls\norm{v\,g}_{r,(A,B)}\quad\forall g\in\Pos{A,B}$$
if and only if
\begin{align*}
\text{either}\quad r\leq s\quad\text{and}\quad& \sup_{A<x<B}\norm{w}_{s,(x,B)}\norm{v^{-1}}_{r',(A,x)}<\infty,\\
\text{or}\quad r>s\quad\text{and}\quad&\norm{\norm{w}_{s,(x,B)}\norm{v^{-1}}_{r',(A,x)}^{\f{r'}{s'}}v(x)^{-\f{r'}{\rh}}}_{\rh,(A,B)}<\infty.
\end{align*}
\item[{\rm(ii)}] Then
$$\norm{w(t)\int_{t}^{B}g}_{s,(A,B)}\ls\norm{v\,g}_{r,(A,B)}\quad\forall g\in\Pos{A,B}$$
if and only if 
\begin{align*}
\text{either}\quad r\leq s\quad\text{and}\quad& \sup_{A<x<B}\norm{w}_{s,(A,x)}\norm{v^{-1}}_{r',(x,B)}<\infty,\\
\text{or}\quad r>s\quad\text{and}\quad&\norm{\norm{w}_{s,(A,x)}\norm{v^{-1}}_{r',(x,B)}^{\f{r'}{s'}}v(x)^{-\f{r'}{\rh}}}_{\rh,(A,B)}<\infty.
\end{align*}
\end{itemize}
\end{theorem}

\begin{proof}[Proof of Lemma~\ref{TLH}.] \textit{Equivalence of} (i) \textit{and} (ii) follows from Theorem \ref{TH} with $v(t)=t^{-\mu+\f1{r'}}a(t)$ and $w(t)=t^{-\mu-\f1s}b(t)$, $\mu>0$, $t\in(A,B)$. Indeed, using Lemma~\ref{LSV} (i), we obtain this way that
\begin{equation}\label{112}
\norm{t^{-\mu-\f1s}b(t)\int_A^th}_{s,(A,B)}\ls\norm{t^{-\mu+\f1{r'}}a(t)h(t)}_{r,(A,B)}\quad\forall h\in\Pos{A,B}
\end{equation}
if and only if $N<\infty$. Condition (i) follows from \eqref{112} on substituting $h(u)=u^{\kappa-1}g(u)$, $u\in(A,B)$, where $\kappa\in\R$ and \mbox{$g\in\Pos{A,B}$},

\textit{Equivalence of} (i) \textit{and} (iv) can be proved analogously as that of (i) and (ii).

\textit{Implications} $({\rm ii})\Rightarrow({\rm iii})$ \textit{and} $({\rm iv})\Rightarrow({\rm v})$ are trivial.

\textit{Implication} $({\rm iii})\Rightarrow({\rm ii})$. Let $g\in\Pos{A,B}$ and 
\begin{equation}\label{133}
f(t)=\int_{t}^{B}g,\quad t\in(A,B).
\end{equation}
Then $f\in\Dupa{A,B}$ and we obtain from (iii) that
\begin{equation}\label{33}
\norm{t^{-\mu-\f1s}b(t)\intt{A}{t}{u^{\nu-1}\left(\int_{u}^{B}g\right)}{u}}_{s,(A,B)}
\ls\norm{t^{-\mu+\nu-\f1r}a(t)\int_{t}^{B}g}_{r,(A,B)}
\end{equation}
for every $g\in\Pos{A,B}$. To estimate the integral on $\LHS{33}$, we use the Fubini's theorem to get
\begin{align}\label{34}
\intt{A}{t}{u^{\nu-1}\intt{u}{B}{g(\tau)}{\tau}}{u}
&\geq\intt{A}{t}{\intt{u}{t}{u^{\nu-1}g(\tau)}{\tau}}{u}
=\intt{A}{t}{\intt{A}{\tau}{u^{\nu-1}g(\tau)}{u}}{\tau}\nonumber\\
&\es\intt{A}{t}{(\tau^{\nu}-A^{\nu})g(\tau)}{\tau}
\end{align}
for all $t\in(A,B)$. If $A\neq0$, i.e. if $(A,B)=(1,\infty)$, then we continue with the estimate as follows:
\begin{align*}
\intt{1}{t}{(\tau^{\nu}-1)g(\tau)}{\tau}
\geq\intt{2}{t}{(\tau^{\nu}-1)g(\tau)}{\tau}
\geq\intt{2}{t}{(\tau^{\nu}-(\tf{\tau}2)^{\nu})g(\tau)}{\tau}
\es\intt{2}{t}{\tau^{\nu}g(\tau)}{\tau}
\end{align*}
for all $t\in(2,\infty)$. This, together with \eqref{34} gives (after simple substitutions)
\begin{equation}\label{1102}
\LHS{33}\gs\norm{t^{-\mu-\f1s}b(t)\intt{A}{t}{u^{\nu}g(u)}{u}}_{s,(A,B)}.
\end{equation}
Now we estimate $\RHS{33}$ from above. We put $\alpha=-\mu+\nu>0$ and $\beta=1$. We are going to apply weighted Hardy inequality (iv) with $\alpha,\beta$ instead of $\mu,\kappa$, respectively, and with $s=r$ and $b=a$, so that $N(r,r,a,a;A,B)<\infty$. Thus, by the equivalence of (i) and (iv), which we have already proved, we get
$$\norm{t^{-\mu+\nu-\f1r}a(t)\int_t^Bg}_{r,(A,B)}
\ls\norm{t^{-\mu+\nu+\f1{r'}}a(t)g(t)}_{r,(A,B)}\quad\forall g\in\Pos{A,B}.$$
This, \eqref{33} and \eqref{1102} give
$$\norm{t^{-\mu-\f1s}b(t)\intt{A}{t}{u^{\nu}g(u)}{u}}_{s,(A,B)}
\ls\norm{t^{-\mu+\nu+\f1{r'}}a(t)g(t)}_{r,(A,B)}\quad\forall g\in\Pos{A,B},$$
which can be rewritten (using the substitution $g(u)=u^{\kappa-\nu-1}h(u)$, $u\in(A,B)$) as (i).

\textit{Implication} $({\rm v})\Rightarrow({\rm iv})$ can be proved similarly as implication ${\rm(iii)}\Rightarrow{\rm(ii)}$. Indeed, using test function \eqref{133} in (iv), we arrive at
\begin{equation}\label{116}
\norm{t^{\mu-\f1s}b(t)\intt{t}{B}{u^{\nu-1}\left(\int_{u}^{B}g\right)}{u}}_{s,(A,B)}
\ls\norm{t^{\mu+\nu-\f1r}a(t)\int_{t}^{B}g}_{r,(A,B)}
\end{equation}
for every $g\in\Pos{A,B}$. The estimate of \LHS{116}, corresponding to \eqref{34}, now takes the form
\begin{align*}
\intt{t}{B}{u^{\nu-1}\intt{u}{B}{g(\tau)}{\tau}}{u}
&=\intt{t}{B}{\intt{t}{\tau}{u^{\nu-1}g(\tau)}{u}}{\tau}
\es\intt{t}{B}{(\tau^{\nu}-t^{\nu})g(\tau)}{\tau}\nonumber\\
&\geq\intt{2t}{B}{(\tau^{\nu}-t^{\nu})g(\tau)}{\tau}
\geq\intt{2t}{B}{(\tau^{\nu}-(\tf{\tau}2)^{\nu})g(\tau)}{\tau}
\es\intt{2t}{B}{\tau^{\nu}g(\tau)}{\tau}
\end{align*}
for all $t\in(A,\tf B2)$. The rest of the proof is analogous to the proof of implication \mbox{${\rm(iii)}\Rightarrow{\rm(ii)}$}.
\end{proof}

\begin{proof}[Proofs of Lemmas~\ref{LRed} and \ref{LRedR}.]
One can repeat the proof of Lemma~\ref{TLH} (the equivalence of (i), (ii), (iii), or (i), (iv), (v), respectively) with $\mu=0$.
\end{proof}


The proof of Lemma~\ref{LSaw} is the most difficult and it will require different approach than the proof of Lemmas~\ref{TLH}, \ref{LRed}, \ref{LRedR}. The problem is that the characterizing conditions for inequality \eqref{36} restricted to non-increasing functions can be actually weaker than the characterizing conditions for the same inequality considered for all non-negative functions (cf. Lemma~\ref{LRedR}~(ii) with $\kappa=0$ and Remark~\ref{Rem}~(iii)). In other words, the restriction of \eqref{36} to non-increasing functions has a significant effect on its characterizing conditions (cf. \cite[p.129]{EOP} and \cite[Remarks~10.5. and 10.8.]{EOP}). This in turn means that one cannot prove the sufficiency of those weaker conditions using Theorem~\ref{TH} and thus, more suitable results are needed - we are going to use the reduction theorem.

Probably the first result of this kind appeared in \cite{Saw}. Sawyer's result can be very well used in our situation; we will, however, use another result by A. Gogatishvili and V. D. Stepanov, which is more recent (and easier to prove). We are going to state it here for an integral operator with general kernel given by
\begin{equation}\label{eq65}
Sg(t)=\intt{A}{B}{k(t,u)g(u)}{u},\quad g\in\Pos{A,B},\quad t\in(A,B),
\end{equation}
where $k$ is non-negative measurable function on $(A,B)\times(A,B)$.

\begin{theorem}\label{AGRed}
Let $1\leq r<\infty$, $0<s\leq\infty$, $v,w\in\Wg(A,B)$ and $V(t)=\int_A^tv$, $t\in(A,B)$. Let $S$ be the integral operator \eqref{eq65} with the kernel $k$. Set $$K(t,u)=\intt{A}{u}{k(t,\tau)}{\tau}\quad\text{and}\quad\con{S}f(t)=\intt{A}{B}{K(t,u)f(u)}{u},\quad t,u\in(A,B).$$
Then
$$\norm{w\,Sf}_{s,(A,B)}\ls\norm{v\,f}_{r,(A,B)}\quad\forall f\in\Dupa{A,B}$$
if and only if
$$\norm{w\,K(\cdot,B)}_{s,(A,B)}\ls\norm{v}_{r,(A,B)}$$
and
$$\norm{w\,\con{S}g}_{s,(A,B)}\ls\norm{v^{1-r}\,Vg}_{r,(A,B)}\quad\forall g\in\Pos{A,B}.$$
\end{theorem}
\begin{proof}
The theorem is an easy consequence of \cite[Theorem~2.1]{AG}.
\end{proof}

We shall also need a characterization of the boundedness of Volterra integral operators defined by
\begin{equation}\label{eq68}
Vg(t)=\intt{0}{t}{k(t,u)g(u)}{u},\quad g\in\Pos{0,\infty},\quad t\in(0,\infty),
\end{equation}
where the kernel $k$ satisfies:
\begin{itemize}
\item[(i)] the function $(t,u)\mapsto k(t,u)$ is non-decreasing in $t$ or non-increasing in $u$;
\item[(ii)] $k(t,u)\geq0$ for all $t>u>0$;
\item[(iii)] $k(t,\tau)\es k(t,u)+k(u,\tau)$ for all $t>u>\tau>0$.
\end{itemize}

\begin{theorem}\label{TVolt}
Let the $V$ be Volterra integral operator \eqref{eq68} with kernel $k$ satisfying {\rm(i)}, {\rm(ii)}, {\rm(iii)}. Suppose $v,w\in\Wg(0,B)$, where $B=1$, or $B=\infty$. 
Then
\begin{equation}\label{1101}
\norm{w\,Vg}_{s,(0,B)}\ls\norm{v\,g}_{r,(0,B)}\quad\forall g\in\Pos{0,B}
\end{equation}
if and only if one of the following conditions hold:
\begin{itemize}
\item[{\rm(i)}] $1<r\leq s<\infty$,
$$\sup_{0<x<B}\norm{w\,k(\cdot,x)}_{s,(x,B)}\norm{v^{-1}}_{r',(0,x)}<\infty,$$
$$\sup_{0<x<B}\norm{w}_{s,(x,B)}\norm{v^{-1}\,k(x,\cdot)}_{r',(0,x)}<\infty;$$
\item[{\rm(ii)}] $1<s<r<\infty$,
$$\norm{v(x)^{-\f{r'}{\rh}}\norm{w\,k(\cdot,x)}_{s,(x,B)}\norm{v^{-1}}^{\f{r'}{s'}}_{r',(0,x)}}_{\rh,(0,B)}<\infty,$$
$$\norm{w(x)^{\f s{\rh}}\norm{w}^{\f sr}_{s,(x,B)}\norm{v^{-1}\,k(x,\cdot)}_{r',(0,x)}}_{\rh,(0,B)}<\infty.$$
\end{itemize}
\end{theorem}
\begin{proof}
If $B=\infty$, then the result can be found in \cite[Theorems 1, 2]{Step}.

In the case $B=1$, we can prove the sufficiency of conditions (i), (ii) by using the theorem with $B=\infty$, $w=\chi_{(0,1)}\con{w}$ and by considering \eqref{1101} for every $g\in\Pos{0,\infty}$, such that $g=0$ on $(1,\infty)$. To prove that conditions (i), (ii) are also necessary in this case, use the same test functions as in \cite{Step}. 
\end{proof}

Finally, we can start with a proof of Lemma~\ref{LSaw}.

\begin{proof}[Proof of Lemma~\ref{LSaw}.]\mbox{}

\textit{Case $r=\infty$.} To prove the necessity of the condition $R_{\infty}<\infty$, we test \eqref{36} by 
$$f(u)=\norm{a}_{\infty,(A,u)}^{-1},\quad u\in(A,B),$$
which is clearly a non-increasing function on $(A,B)$. In this way, we obtain
\begin{align}\label{160}
\norm{t^{-\f1s}b(t)\intt{t}{B}{u^{-1}\norm{a}_{\infty,(A,u)}^{-1}}{u}}_{s,(A,B)}
&\ls\norm{a(t)\norm{a}_{\infty,(A,t)}^{-1}}_{\infty,(A,B)}\nonumber\\
&\ls\norm{\norm{a}_{\infty,(A,t)}\norm{a}_{\infty,(A,t)}^{-1}}_{\infty,(A,B)}=1,
\end{align}
which we wanted to show.

To prove the sufficiency, we use $R_{\infty}<\infty$ (i.e. \eqref{160}) and the monotonicity of $f$ to get
\begin{align*}
\LHS{36}
&=\norm{t^{-\f1s}b(t)\intt{t}{B}{u^{-1}\norm{a}_{\infty,(A,u)}^{-1}\;\norm{a}_{\infty,(A,u)}f(u)}{u}}_{s,(A,B)}\\
&\leq\norm{t^{-\f1s}b(t)\norm{\norm{a}_{\infty,(A,u)}f(u)}_{\infty,(t,B)}\intt{t}{B}{u^{-1}\norm{a}_{\infty,(A,u)}^{-1}}{u}}_{s,(A,B)}\\
&\ls\norm{\norm{a}_{\infty,(A,t)}f(t)}_{\infty,(A,B)}
\leq\norm{\norm{af}_{\infty,(A,t)}}_{\infty,(A,B)}
=\RHS{36},
\end{align*}
hence the case $r=\infty$ is proved.

In the remaining cases, Theorem~\ref{AGRed} with $k(t,u)=\chi_{(t,B)}(u)u^{-1}$, $v(t)=t^{-\f1r}a(t)$ and $w(t)=t^{-\f1s}b(t)$, $t,u\in(A,B)$, yields that \eqref{36} holds for all $f\in\Dupa{A,B}$ if and only if
\begin{equation}\label{1114}
\norm{t^{-\f1s}b(t)\log\tf Bt}_{s,(A,B)}\ls\norm{t^{-\f1r}a(t)}_{r,(A,B)}
\end{equation}
and
\begin{equation}\label{1113}
\norm{t^{-\f1s}b(t)\intt{t}{B}{g(u)\log\tf ut}{u}}_{s,(A,B)}\!\!\!\!\!\!\!\!\ls\norm{t^{\f1{r'}}a(t)^{-\f r{r'}}V(t)g(t)}_{r,(A,B)}\;\forall g\in\Pos{A,B},
\end{equation}
where $V(t)=\intt{A}{t}{u^{-1}a(u)^r}{u}$, $t\in(A,B)$. Condition \eqref{1114} translates as \eqref{A} if $B=\infty$. When $B=1$, then \eqref{1114} means that if $\RHS{1114}$ is finite, then $\LHS{1114}$ is as well. We will now show that this is, in fact, a consequence of $R_{\infty}(r,s,a,b,0,1)<\infty$. Indeed, this is obvious in all cases but $1<r<s=\infty$, i.e. when $R_{\infty}$ is defined only by $R_2$. In this case, using the Lemma~\ref{L} and the assumption $\RHS{1114}<\infty$, we obtain
\begin{align*}
\norm{t^{-\f1{r'}}a(t)^{\f r{r'}}V(t)^{-1}\log\tf tx}_{r',(x,1)}
&\geq\norm{t^{-\f1{r'}}a(t)^{\f r{r'}}V(t)^{-1}}_{r',(\sqrt{x},1)}\log\tf{\sqrt{x}}{x}\\
&\es\norm{t^{-\f1r}a(t)}_{r,(0,\sqrt{x})}^{-1}\log\tf1x
\gs\log\tf1x
\end{align*} 
for all $x\in(0,\tf12)$. Thus, if $R_{\infty}(r,\infty,a,b;0,1)<\infty$, then
\begin{align*}
\infty&>R_{\infty}(r,\infty,a,b;0,1)=R_2(r,\infty,a,b;0,1)
\nonumber\\
&\qquad=\sup_{0<x<1}\norm{b}_{\infty,(0,x)}\norm{t^{-\f1{r'}}a(t)^{\f r{r'}}V(t)^{-1}\log\tf tx}_{r',(x,1)}\gs\sup_{0<x<1}\norm{b}_{\infty,(0,x)}\log\tf1x.
\end{align*}
From that we finally get
$$\LHS{1114}=\norm{b(t)\log\tf1t}_{\infty,(0,1)}
\ls\norm{\norm{b}_{\infty,(0,t)}\log\tf1t}_{\infty,(0,1)}<\infty.$$

Now it remains to prove that, under condition \eqref{A}, inequality \eqref{1113} holds if and only if $R_{\infty}(r,s,a,b;A,B)<\infty$.

\textit{Case $1<r,s<\infty$.} Note that, by the duality (or more precisely, by the sharp H\"{o}lder's inequality, cf. \cite[Chapter 1, Theorem 2.5.]{BS}), inequality \eqref{1113} holds if and only if
\begin{equation}\label{1115}
\norm{t^{-\f1{r'}}a(t)^{\f r{r'}}V(t)^{-1}\intt{A}{t}{g(u)\log\tf tu}{u}}_{r',(A,B)}\ls
\norm{t^{\f1s}b(t)^{-1}g(t)}_{s',(A,B)}
\end{equation}
holds for all $g\in\Pos{A,B}$. Now we apply Theorem~\ref{TVolt} with $w(t)=t^{-\f1{r'}}a(t)^{\f r{r'}}V(t)^{-1}$, $v(t)=t^{\f1s}b(t)^{-1}$, $t\in(A,B)$, with $r,s$ replaced by $s',r'$ and we use Lemma~\ref{L} to get that \eqref{1115} holds for every $g\in\Pos{A,B}$ if and only if $R_{\infty}<\infty$.

\textit{Case $r=1$.} We can rewrite \eqref{1113} as
\begin{equation}\label{140}
\norm{\intt{A}{B}{k(t,u)g(u)}{u}}_{s,(A,B)}\ls\norm{g}_{1,(A,B)}\quad\forall g\in\Pos{A,B},
\end{equation}
where
\begin{equation}\label{125}
k(t,u)=t^{-\f1s}b(t)V(u)^{-1}\chi_{(t,B)}(u)\log\tf ut,\quad t,u\in(A,B).
\end{equation}

We claim that \eqref{140} holds if and only if
\begin{equation}\label{138}
\sup_{A<u<B}\norm{k(\cdot,u)}_{s,(A,B)}<\infty.
\end{equation}
Indeed, the general result for arbitrary kernels \cite[Chapter XI, Theorem 4]{Ak} implies that \eqref{140} is equivalent to $\esssup_{A<x<B}\norm{k(\cdot,u)}_{r',(A,B)}<\infty$, which, in our case, is equivalent to \eqref{138}. However, let us also give an explicit proof of this claim, using the properties of $k$.

To prove the sufficiency of \eqref{138}, we take $g,h\in\Meas^+(A,B)$ and write
\begin{align}\label{1010}
\intt{A}{B}{\intt{A}{B}{k(t,u)g(u)}{u}\;h(t)}{t}
&=\intt{A}{B}{\intt{A}{B}{k(t,u)h(t)}{t}\;g(u)}{u}\nonumber\\
&\leq\intt{A}{B}{\norm{k(\cdot,u)}_{s,(A,B)}\norm{h}_{s',(A,B)}g(u)}{u}\nonumber\\
&\ls\esssup_{A<u<B}\norm{k(\cdot,u)}_{s,(A,B)}\norm{h}_{s',(A,B)}\norm{g}_{1,(A,B)}.
\end{align}
Inequality \eqref{140} then follows from \eqref{1010} by taking the supremum over all $h$ with $\norm{h}_{s',(A,B)}\leq1$, using the sharp H\"{o}lder's inequality and \eqref{138}.

Conversely, suppose that \eqref{140} holds. Fix $x\in(A,B)$ and test \eqref{140} with
$$g_{x,n}=n\chi_{(x,x+\f1n)},\quad n\in\N,$$
to get
\begin{equation}\label{1003}
\norm{n\intt{x}{x+\f1n}{k(\cdot,u)}{u}}_{r',(A,B)}\ls1\quad\forall n\in\N.
\end{equation}
We see from \eqref{125} that $k$ is continuous in the second variable in $(A,B)$. Therefore,
$$\lim_{n\to\infty}n\intt{x}{x+\f1n}{k(t,u)}{u}=k(t,x)\quad\forall t\in(A,B)$$
by the fundamental theorem of calculus. Thus, using \eqref{1003} and Fatou's lemma, we obtain
\begin{equation}\label{1005}
1\gs\liminf_{n\to\infty}\,\norm{n\intt{x}{x+\f1n}{k(\cdot,u)}{u}}_{r',(A,B)}
\geq\norm{k(\cdot,x)}_{r',(A,B)}.
\end{equation}
Since the multiplicative constant in \eqref{1005} does not depend on $x$, we get \eqref{138}.

It is easy to see that the condition \eqref{138} with $k$ given by \eqref{125} coincides with $R_{\infty}<\infty$, hence the proof of the case $r=1$ is finished.

\textit{Case $1=s<r<\infty$.} Instead of \eqref{1113}, we will characterize its equivalent dual version \eqref{1115}, which can be rewritten as
\begin{equation}\label{1200}
\norm{t^{-\f1{r'}}a(t)^{\f r{r'}}V(t)^{-1}\intt{A}{t}{u^{-1}b(u)g(u)\log\tf tu}{u}}_{r',(A,B)}\ls\norm{g}_{\infty,(A,B)},
\end{equation}
for all $g\in\Pos{A,B}$. It is now obvious that the condition
$$\norm{t^{-\f1{r'}}a(t)^{\f r{r'}}V(t)^{-1}\intt{A}{t}{u^{-1}b(u)\log\tf tu}{u}}_{r',(A,B)}\ls1$$
is both sufficient and necessary for \eqref{1200} and also that it coincides with $R_{\infty}<\infty$ in this case.

\textit{Case $1<r<s=\infty$.} Now \eqref{1115} can be rewritten as
$$\norm{\intt{A}{B}{k(t,u)g(u)}{u}}_{r',(A,B)}\ls\norm{g}_{1,(A,B)}\quad\forall g\in\Pos{A,B},$$
where
\begin{equation}\label{1201}
k(t,u)=t^{-\f1{r'}}a(t)^{\f r{r'}}V(t)^{-1}b(u)\chi_{(A,t)}(u)\log\tf tu,\quad t,u\in(A,B),
\end{equation}
and we can use the same technique of proof as in the case $r=1$. There is one slight difference that instead of $b$ itself we need to take its continuous representation (from Lemma~\ref{LSV}~(i)). This way, we obtain condition \eqref{138} again, only with $r'$ instead of $s$. To see that this condition coincides with $R_{\infty}<\infty$, we use \eqref{1201} to get
\begin{equation}\label{1202}
\infty>\sup_{A<x<B}\norm{k(\cdot,x)}_{r',(A,B)}
=\sup_{A<x<B}b(x)f(x),
\end{equation}
where $x\mapsto f(x):=\norm{t^{-\f1{r'}}a(t)^{\f r{r'}}V(t)^{-1}\log\tf tx}_{r',(x,B)}$ is a decreasing function in $(A,B)$. Now observe that
$$b(x)f(x)
\ls\norm{b}_{\infty,(A,x)}f(x)\leq\norm{bf}_{\infty,(A,x)}
\leq\norm{bf}_{\infty,(A,B)}\quad\forall x\in(A,B),$$
and hence, using \eqref{1202}, we obtain
$$\infty>\sup_{A<x<B}\norm{k(\cdot,x)}_{r',(A,B)}
\es\norm{bf}_{\infty,(A,B)}
\es\sup_{0<x<\infty}\norm{b}_{\infty,(A,x)}f(x),$$
which is precisely $R_{\infty}<\infty$. This finishes the proof of the case $1<r<s=\infty$ and of the lemma.
\end{proof}

\section{The proofs of the main results}\label{S5}
First we shall prove the following simple lemma.

\begin{lemma}\label{LNer}
Let $r,s\in[1,\infty]$, $a,b\in\SV(A,B)$ and suppose that \eqref{210} is satisfied. Then
$$N(r,s,a,b;A,B)\ls\min\left(L(r,s,a,b;A,B),R(r,s,a,b;A,B),R_{\infty}(r,s,a,b;A,B)\right).$$
\end{lemma}
\begin{proof}
The inequalities $N\ls L$ and $N\ls R$ follow easily from Lemma~\ref{LSV}~(ii). The inequality $N\ls R_{\infty}$ is a consequence of the estimates
\begin{align*}
\norm{t^{-\f1s}b(t)\log\tf xt}_{s,(A,x)}
&\geq\norm{t^{-\f1s}b(t)}_{s,(A,\tf x2)}\log2 &\forall x\in(2A,B),\\
\norm{t^{-\f1{r'}}a(t)^{\f r{r'}}V(t)^{-1}\log\tf tx}_{r',(x,B)}
&\geq\norm{t^{-\f1{r'}}a(t)^{\f r{r'}}V(t)^{-1}}_{r',(2x,B)}\log2 &\forall x\in(A,\tf B2),
\end{align*}
Lemma~\ref{L}, assumption \eqref{210} and of Lemma~\ref{LSV}~(ii).
\end{proof}

\begin{proof}[Proof of Theorem~\ref{IntLim}.]
First of all, we are going to show that the assumptions
\begin{equation}\label{1104a}
T\in\JW(p_1,q_1;p_2,q_2)
\end{equation}
and
\begin{equation}\label{1104b}
T\in\LBl(p_1,r_1;m)\cap\LBr(p_2,q_2;m)
\end{equation}
imply that
\begin{equation}\label{1105}
\norm{(Tg)^*}_B\ls\norm{g^*}_A\quad\forall g\in\Meas(R_1,\mu_1)
\end{equation}
is equivalent to
\begin{equation}\label{1106}
\norm{S_{\sigma}f}_B\ls\norm{f}_A\quad\forall\Dupa{0,\infty},
\end{equation}
where $\norm{\cdot}_A$ and $\norm{\cdot}_B$ are some rearrangement-invariant quasi-norms on $\Pos{0,\infty}$ (that we specify later on). Indeed, using \eqref{1104b} and \eqref{1105}, we get
$$\norm{S_{\sigma}f}_B\ls\norm{(Tg)^*}_B\ls\norm{g^*}_A=\norm{f}_A\quad\forall f\in\Dupa{0,\infty},$$
where $g\in\Meas(R_1,\mu_1)$ is equimeasurable with $f$. To prove the opposite implication, use \eqref{1104a}  and \eqref{1106} to obtain
$$\norm{(Tg)^*}_B\ls\norm{S_{\sigma}g^*}_B\ls\norm{g^*}_A\quad\forall g\in\Meas(R_1,\mu_1).$$

Thus, the question of the boundedness of $T$ is reduced to the characterization of \eqref{1106}.

\textit{Case} (i). Now $\norm{\cdot}_A=\norm{\cdot}_{(p_1,r_1)+(p_2,r_2);a}$ and $\norm{\cdot}_B=\norm{\cdot}_{(q_1,s_1)+(q_2,s_2);b}$. If we make a temporary assumption that $m>0$ (i.e. $q_1<q_2$) and use the substitution $\tau=t^m$, then
\begin{align}\label{1233}
&\norm{S_{\sigma}f}_{(q_1,s_1)+(q_2,s_2);b}
=\norm{t^{\f1{q_1}-\f1{s_1}}b(t)S_{\sigma}f(t)}_{s_1,(0,1)}\!\!\!\!\!\!\!\!\!+\norm{t^{\f1{q_2}-\f1{s_2}}b(t)S_{\sigma}f(t)}_{s_2,(1,\infty)}\nonumber\\
&\quad\es\norm{t^{-\f1{s_1}}b(t)\intt{0}{t^m}{u^{\f1{p_1}-1}f(u)}{u}}_{s_1,(0,1)}
\!\!\!\!\!\!\!\!\!+\norm{t^{\f1{q_1}-\f1{q_2}-\f1{s_1}}b(t)\intt{t^m}{\infty}{u^{\f1{p_2}-1}f(u)}{u}}_{s_1,(0,1)}\nonumber\\
&\qquad+\norm{t^{\f1{q_2}-\f1{q_1}-\f1{s_2}}b(t)\intt{0}{t^m}{u^{\f1{p_1}-1}f(u)}{u}}_{s_2,(1,\infty)}
\!\!\!\!\!\!\!\!\!+\norm{t^{-\f1{s_2}}b(t)\intt{t^m}{\infty}{u^{\f1{p_2}-1}f(u)}{u}}_{s_2,(1,\infty)}\nonumber\\
&\quad\es\norm{t^{-\f1{s_1}}b_*(t)\intt{0}{t}{u^{\f1{p_1}-1}f(u)}{u}}_{s_1,(0,1)}
\!\!\!\!\!\!\!\!\!+\norm{t^{\f1{p_1}-\f1{p_2}-\f1{s_1}}b_*(t)\intt{t}{\infty}{u^{\f1{p_2}-1}f(u)}{u}}_{s_1,(0,1)}\nonumber\\
&\qquad+\norm{t^{\f1{p_2}-\f1{p_1}-\f1{s_2}}b_*(t)\intt{0}{t}{u^{\f1{p_1}-1}f(u)}{u}}_{s_2,(1,\infty)}
\!\!\!\!\!\!\!\!\!+\norm{t^{-\f1{s_2}}b_*(t)\intt{t}{\infty}{u^{\f1{p_2}-1}f(u)}{u}}_{s_2,(1,\infty)}\nonumber\\
&=:N_1+N_2+N_3+N_4
\end{align}
for all $f\in\Dupa{0,\infty}$. Now observe that for $m<0$, the role of the intervals $(0,1)$ and $(1,\infty)$ in the computation above is interchanged at the initial stage (cf. Definition~\ref{DSum}), but then the substitution swaps the intervals once more. Therefore, the resulting expression is the same and the assumption $m>0$ can be removed. In the rest of the proof we apply the weighted inequalities of Section~\ref{S4} to show that
\begin{align}\label{1108}
N_1+N_2+N_3+N_4&\ls\norm{t^{\f1{p_1}-\f1{r_1}}a(t)f(t)}_{r_1,(0,1)}+\norm{t^{\f1{p_2}-\f1{r_2}}a(t)f(t)}_{r_2,(1,\infty)}\nonumber\\
&=\norm{f}_{(p_1,r_1)+(p_2,r_2);a}
\end{align}
for every $f\in\Dupa{0,\infty}$ if and only if
\begin{equation}\label{1107}
L(r_1,s_1,a,b_*;0,1)+R(r_2,s_2,a,b_*;1,\infty)<\infty.
\end{equation}

Lemma~\ref{LRed} with $\nu=\f1{p_1}$ implies that
\begin{equation}\label{1220}
N_1\ls\norm{t^{\f1{p_1}-\f1{r_1}}a(t)f(t)}_{r_1,(0,1)}\leq\norm{f}_{(p_1,r_1)+(p_2,r_2);a}\quad\forall f\in\Dupa{0,\infty}
\end{equation}
if and only if $L(r_1,s_1,a,b_*;0,1)<\infty$. Similarly, we get from Lemma~\ref{LRedR} with $\nu=\f1{p_2}$ that
\begin{equation}\label{1221}
N_4\ls\norm{t^{\f1{p_2}-\f1{r_2}}a(t)f(t)}_{r_2,(1,\infty)}\leq\norm{f}_{(p_1,r_1)+(p_2,r_2);a}\quad\forall f\in\Dupa{0,\infty}
\end{equation}
if only if $R(r_2,s_2,a,b_*;1,\infty)<\infty$. Now we estimate the expressions $N_2$ and $N_3$. By Lemma~\ref{LNer}, condition \eqref{1107} implies
\begin{equation}\label{1110}
N(r_1,s_1,a,b_*;0,1)+N(r_2,s_2,a,b_*;1,\infty)<\infty
\end{equation}
and also (using the properties of s.v. functions)
\begin{equation}\label{1109}
\norm{t^{-\f1{r_1'}}a(t)^{-1}}_{r_1',(0,1)}
+\norm{t^{-\f1{r_2'}}a(t)^{-1}}_{r_2',(1,\infty)}<\infty.
\end{equation}
Thus, using \eqref{1110}, Lemma~\ref{TLH} with $\mu=\f1{p_1}-\f1{p_2},\kappa=\f1{p_2}$, Lemma~\ref{LSV}~(i), H\"{o}lder's inequality and \eqref{1109}, we obtain
\begin{align}\label{1223}
N_2&\es\norm{t^{\f1{p_1}-\f1{p_2}-\f1{s_1}}b_*(t)\intt{t}{1}{u^{\f1{p_2}-1}f(u)}{u}}_{s_1,(0,1)}\nonumber\\
&\mkern100mu+\norm{t^{\f1{p_1}-\f1{p_2}-\f1{s_1}}b_*(t)}_{s_1,(0,1)}\intt{1}{\infty}{u^{\f1{p_2}-1}f(u)}{u}\nonumber\\
&\ls\norm{t^{\f1{p_1}-\f1{r_1}}a(t)f(t)}_{r_1,(0,1)}
\!\!+\norm{t^{\f1{p_2}-\f1{r_2}}a(t)f(t)}_{r_2,(1,\infty)}\norm{t^{-\f1{r_2'}}a(t)^{-1}}_{r_2',(1,\infty)}\nonumber\\
&\es\norm{f}_{(p_1,r_1)+(p_2,r_2);a}
\end{align}
for all $f\in\Dupa{0,\infty}$. Analogically, we can estimate $N_3$:
\begin{align}\label{1224}
N_3&\es\norm{t^{\f1{p_2}-\f1{p_1}-\f1{s_2}}b_*(t)\intt{1}{t}{u^{\f1{p_1}-1}f(u)}{u}}_{s_2,(1,\infty)}\nonumber\\
&\mkern100mu+\norm{t^{\f1{p_2}-\f1{p_1}-\f1{s_1}}b_*(t)}_{s_2,(1,\infty)}\intt{0}{1}{u^{\f1{p_1}-1}f(u)}{u}\nonumber\\
&\ls\norm{t^{\f1{p_2}-\f1{r_2}}a(t)f(t)}_{r_2,(1,\infty)}
+\norm{t^{\f1{p_1}-\f1{r_1}}a(t)f(t)}_{r_1,(0,1)}\norm{t^{-\f1{r_1'}}a(t)^{-1}}_{r_1',(0,1)}\nonumber\\
&\es\norm{f}_{(p_1,r_1)+(p_2,r_2);a}
\end{align}
for every $f\in\Dupa{0,\infty}$. The inequality \eqref{1108} then follows from \eqref{1220}, \eqref{1221}, \eqref{1223} and \eqref{1224}, hence the proof of part (i) is complete.

\textit{Case} (ii). We can proceed in the same way as in the case (i) to obtain that
$$T:L_{p_1,r_1;a}\cap L_{p_2,r_2;a}\lra L_{q_1,s_1;b}\cap L_{q_2,s_2;b}$$
holds if and only if the inequality
\begin{align}\label{1230}
N_1&+N_2+N_3+N_4:=\nonumber\\
&\norm{t^{-\f1{s_1}}b_*(t)\intt{0}{t}{u^{\f1{p_1}-1}f(u)}{u}}_{s_1,(1,\infty)}
\!\!\!\!\!\!\!\!\!+\norm{t^{\f1{p_1}-\f1{p_2}-\f1{s_1}}b_*(t)\intt{t}{\infty}{u^{\f1{p_2}-1}f(u)}{u}}_{s_1,(1,\infty)}\nonumber\\
&\quad+\norm{t^{\f1{p_2}-\f1{p_1}-\f1{s_2}}b_*(t)\intt{0}{t}{u^{\f1{p_1}-1}f(u)}{u}}_{s_2,(0,1)}
\!\!\!\!\!\!\!\!\!+\norm{t^{-\f1{s_2}}b_*(t)\intt{t}{\infty}{u^{\f1{p_2}-1}f(u)}{u}}_{s_2,(0,1)}\nonumber\\
&\ls\norm{t^{\f1{p_1}-\f1{r_1}}a(t)f(t)}_{r_1,(1,\infty)}+\norm{t^{\f1{p_2}-\f1{r_2}}a(t)f(t)}_{r_2,(0,1)}=\norm{f}_{(p_1,r_1)\cap(p_2,r_2);a}
\end{align}
holds for all $f\in\Dupa{0,\infty}$ (the only difference from the corresponding inequality in the case (i) is that the intervals $(0,1)$ and $(1,\infty)$ were interchanged, cf. Definition~\ref{DSum}). As in the case (i), the (parts of) terms $N_1$ and $N_4$ represent the limiting case of interpolation and hence, the Lemmas~\ref{LRed}, \ref{LRedR} imply that the condition
\begin{equation}\label{1231}
L(r_1,s_1,a,b_*;1,\infty)+R(r_2,s_2,a,b_*,0,1)<\infty
\end{equation}
is necessary for \eqref{1230} to hold for every $f\in\Dupa{0,\infty}$. Now, it remains to prove that \eqref{1231} is also sufficient to estimate all the remaining terms in $\LHS{1230}$ by $\RHS{1230}$.

Note that \eqref{1231} implies
\begin{equation}\label{1232}
\norm{t^{-\f1{s_1}}b_*(t)}_{s_1,(1,\infty)}+\norm{t^{-\f1{s_2}}b_*(t)}_{s_2,(0,1)}<\infty.
\end{equation}
The expression $N_1$ contains the following term, which, using \eqref{1232}, H\"{o}lder's inequality and Lemma~\ref{LSV}~(i), can be estimated as
\begin{align*}
\norm{t^{-\f1{s_1}}b_*(t)\intt{0}{1}{u^{\f1{p_1}-1}f(u)}{u}}_{s_1,(1,\infty)}
\!\!\!\!\!\!\!&\ls\norm{t^{\f1{p_2}-\f1{r_2}}a(t)f(t)}_{r_2,(0,1)}\norm{t^{\f1{p_1}-\f1{p_2}-\f1{r_2'}}a(t)^{-1}}_{r_2',(0,1)}\\
&\ls\norm{t^{\f1{p_2}-\f1{r_2}}a(t)f(t)}_{r_2,(0,1)}.
\end{align*}
The corresponding term in expression $N_4$ can be estimated in the same way as
\begin{align*}
\norm{t^{-\f1{s_2}}b_*(t)\intt{1}{\infty}{u^{\f1{p_2}-1}\!\!f(u)}{u}}_{s_2,(0,1)}
\!\!\!\!\!\!\!\!\!&\ls\norm{t^{\f1{p_1}-\f1{r_1}}a(t)f(t)}_{r_1,(1,\infty)}\norm{t^{\f1{p_2}-\f1{p_1}-\f1{r_1'}}a(t)^{-1}}_{r_1',(1,\infty)}\\
&\ls\norm{t^{\f1{p_1}-\f1{r_1}}a(t)f(t)}_{r_1,(1,\infty)}.
\end{align*}
By Lemma~\ref{LNer}, \eqref{1231} implies
$$N(r_1,s_1,a,b_*,1,\infty)+N(r_2,s_2,a,b_*,0,1)<\infty$$
and hence, using Lemma~\ref{TLH}, the remaining (non-limiting) terms $N_2$, $N_3$ can be estimated by $\RHS{1230}$ as well.
\end{proof}

\begin{proof}[Proof of Theorem~\ref{IntHil}.] One can repeat the proof of Theorem~\ref{IntLim} with $p_2=\infty$, the only difference being that we use Lemma~\ref{LSaw} instead of Lemma~\ref{LRedR}.
\end{proof}

The proofs of Theorems~\ref{IntLim} and \ref{IntHil} contain all the possible difficulties which one can encounter. More precisely, in the proofs of the remaining theorems of Section~\ref{S3} the inequality \eqref{1233} (or \eqref{1230}) will have just two terms - either $N_1+N_2$, or $N_3+N_4$. Therefore, we can obtain the remaining proofs as fragments of the proof of Theorem~\ref{IntLim}.

\section{The optimality of results}\label{S6}
Let $X,Y,W,Z$ be LK spaces, or their sum, or intersection in the sense of Definition~\ref{DSum} and let $T$ be a quasilinear operator acting between $X$ and $Y$. We say that the result
$$T:X\lra Y$$
is optimal in the scale of LK spaces if
$$Y\embl Z\text{ for every }Z\text{ satisfying }T:X\lra Z$$
and
$$W\embl X\text{ for every }W\text{ satisfying }T:W\lra Y.$$
Embeddings of LK spaces are characterized by the following lemma.

\begin{lemma}\label{Temb}
Let $p,q,r,s\in(0,\infty]$ and $a,b\in\SV(A,B)$. Then
\begin{equation}\label{146}
\norm{t^{\f1q-\f1s}b(t)f(t)}_{s,(A,B)}\ls\norm{t^{\f1p-\f1r}a(t)f(t)}_{r,(A,B)}\quad\forall f\in\Dupa{A,B},
\end{equation}
if and only if one of the following conditions hold\,{\rm:}
\begin{itemize}
\item[\rm{(i)}] $(A,B)=(0,1)$, $p>q${\rm;}
\item[\rm{(ii)}] $p=q$, $0<r\leq s\leq\infty$,
\begin{equation}\label{142}\sup_{A<x<B}\norm{t^{\f1{p}-\f1s}b(t)}_{s,(A,x)}\norm{t^{\f1{p}-\f1r}a(t)}_{r,(A,x)}^{-1}<\infty{\rm;}
\end{equation}
\item[{\rm(iii)}] $p=q$, $0<s<r\leq\infty$,
\begin{equation}\label{143}
\norm{x^{\f s{\rh}\f 1p-\f1{\rh}}b(x)^{\f s{\rh}}\norm{t^{\f1p-\f1s}b(t)}_{s,(A,x)}^{\f sr}\norm{t^{\f1p-\f1r}a(t)}_{r,(A,x)}^{-1}}_{\rh,(A,B)}<\infty.
\end{equation}
\end{itemize}
When $p<\infty$, conditions \eqref{142} and \eqref{143} can be simplified to
\begin{equation}\label{148}
N(r,s,a,b;A,B)<\infty.
\end{equation}
\end{lemma}
\begin{proof}
The simplification of \eqref{142} and \eqref{143} in the case $p<\infty$ follows easily from Lemma~\ref{LSV}~(i).

\textit{Case $p>q$.} See \cite[Theorem 3.4]{N}.

\textit{Case $p=q$, $0<r,s<\infty$.} The inequality \eqref{146} can be further rewritten as
\begin{equation}\label{145}
\sup_{f\in\Dupa{A,B}}\f{\norm{wf}_{s,(A,B)}}{\norm{vf}_{r,(A,B)}}<\infty,
\end{equation}
where $v(t)=t^{\f1p-\f1r}a(t)$ and $w(t)=t^{\f1p-\f1s}b(t)$, $t\in(A,B)$. The problem of characterization of \eqref{145} with general weights is fully resolved for $0<r,s<\infty$ and leads directly to conditions \eqref{142} and \eqref{143}. The first result of this kind is due to E. Sawyer for the range $1<r,s<\infty$ (he applied his reduction theorem for the identity operator - see \cite[p.148]{Saw}). This result was extended to $0<r,s<\infty$ by, for example, M. Carro and J. Soria, or V. D. Stepanov in \cite[Proposition~1]{Step2}, who also provided estimates (independent of $v,w$) for \LHS{145}. However, Stepanov's proof relies on the approximation of non-increasing functions by absolutely continuous functions, which was left unjustified. For a rigorous and yet very elegant treatment of this topic, we refer to \cite[Section 2]{Sin}.

We are going to prove the cases which are missing in the literature cited above, that is, cases where $r=\infty$ or $s=\infty$.

\textit{Case $p=q$, $0<r\leq s=\infty$.} To prove the necessity of \eqref{142} for \eqref{146}, it is enough to test \eqref{146} with $f=\chi_{(A,x)}$, where $x\in(A,B)$. 

Now we prove the sufficiency of \eqref{142} for \eqref{146}. Using the estimate
$$t^{\f1p}b(t)\ls\norm{u^{\f1p}b(u)}_{\infty,(A,t)}\quad\forall t\in(A,B)$$
together with \eqref{142} and the monotonicity of $f$, we obtain
\begin{align*}
\norm{t^{\f1p}b(t)f(t)}_{\infty,(A,B)}
&\ls\norm{\norm{u^{\f1p}b(u)}_{\infty,(A,t)}f(t)}_{\infty,(A,B)}
\ls\norm{\norm{u^{\f1p-\f1r}a(u)}_{r,(A,t)}f(t)}_{\infty,(A,B)}\\
&\leq\norm{\norm{u^{\f1p-\f1r}a(u)f(u)}_{r,(A,t)}}_{\infty,(A,B)}
=\RHS{146}
\end{align*}
for every $f\in\Dupa{A,B}$, which proves \eqref{146}.

\textit{Case $p=q<\infty$, $0<s<r=\infty$.} The necessity of \eqref{148} follows by testing \eqref{146} with $f(t)=t^{-\f1p}a(t)^{-1}$, $t\in(A,B)$.
For the sufficiency, we use \eqref{148} to obtain
\begin{align*}
\LHS{146}
&=\norm{t^{\f1p-\f1s}b(t)f(t)}_{s,(A,B)}
=\norm{t^{-\f1s}b(t)a(t)^{-1}\;t^{\f1p}a(t)f(t)}_{s,(A,B)}\\
&\leq\norm{t^{-\f1s}b(t)a(t)^{-1}}_{s,(A,B)}\norm{t^{\f1p}a(t)f(t)}_{\infty,(A,B)}
\ls\RHS{146}
\end{align*}
for every $f\in\Dupa{A,B}$.

\textit{Case $p=q=\infty$, $0<s<r=\infty$.} To prove the necessity of \eqref{143} for \eqref{146}, we test \eqref{146} with $f(t)=\norm{a}^{-1}_{\infty,(A,t)}$, $t\in(A,B)$. In this way, we get
\begin{align}\label{147}
\norm{t^{-\f1s}b(t)\norm{a}_{\infty,(A,t)}^{-1}}_{s,(A,B)}\ls\norm{a(t)\norm{a}_{\infty,(A,t)}^{-1}}_{\infty,(A,B)}
\ls1,
\end{align}
which is indeed \eqref{143} with $p=r=\infty$.

To show the sufficiency, we use \eqref{147} and the monotonicity of $f$ to obtain
\begin{align*}
\norm{t^{-\f1s}b(t)f(t)}_{s,(A,B)}
&\leq\norm{t^{-\f1s}b(t)\norm{a}^{-1}_{\infty,(A,t)}}_{s,(A,B)}\norm{\norm{a}_{\infty,(A,t)}f(t)}_{\infty,(A,B)}\\
&\ls\norm{\norm{af}_{\infty,(A,t)}}_{\infty,(A,B)}\!=\RHS{146}
\end{align*}
for every $f\in\Dupa{A,B}$ and thus, the proof is finished.
\end{proof}

Since every non-increasing function on $(A,B)$ arises as a (restriction of) decreasing rearrangement of a function from $\Meas(R_1,\mu_1)$ (see \cite[p.86, Corollary 7.8.]{BS}), the Lemma~\ref{Temb} characterizes the embedding $L_{p,r;a}\embl L_{q,s;b}$. For the embeddings of the sums and intersections of the LK spaces in the sense of Definition~\ref{DSum}, we apply Lemma~\ref{Temb} on the two parts of the corresponding quasi-norm separately. 

Now we turn our attention to the optimality itself, which, in the non-limiting case, is simple to describe. To save some space, we illustrate the idea only in the setting $\mu_1(R_1)=\mu_2(R_2)=\infty$ without considering sums and intersections of spaces. For the other settings analogous assertions to the following one hold as well and proofs are similar.

\begin{theorem}\label{Opt}
Let the assumptions of Theorem~\ref{Int} be satisfied. Then
\begin{equation}\label{181}
T:L_{p,s;b_*}\lra L_{q,s;b}
\end{equation}
is an optimal result in the scale of LK spaces.
\end{theorem}
\begin{proof}
It follows immediately from Theorem~\ref{Int} that \eqref{181} holds. Next we shall prove that the choice of the target space $L_{q,s;b}$ is optimal in the scale of LK spaces (the optimality of the source space can be proved analogously). Suppose that
\begin{equation}\label{193}
T:L_{p,s;b_*}\lra L_{Q,R;\lambda}
\end{equation}
for some $Q,R\in[1,\infty]$ and $\lambda\in \SV$. This together with $T\in\LBl(p_1,q_1;m)$ gives, for all $f\in\Dupa{0,\infty}$, that
$$\norm{t^{(\f1Q-\f1q)+(\f1q-\f1{q_1})-\f1R}\la(t)\intt{0}{t^m}{u^{\f1{p_1}-1}f(u)}{u}}_{R,(0,\infty)}
\ls\norm{t^{\f1p-\f1s}b_*(t)f(t)}_{s,(0,\infty)},$$
which can be rewritten (using the change of variables) as
\begin{equation}\label{195}
\norm{t^{\gamma+\f1p-\f1{p_1}-\f1R}\la_*(t)\intt{0}{t}{u^{\f1{p_1}-1}f(u)}{u}}_{R,(0,\infty)}
\ls\norm{t^{\f1p-\f1s}b_*(t)f(t)}_{s,(0,\infty)},
\end{equation}
where $\gamma=\f1m(\f1Q-\f1q)$. First we are going to show that \eqref{195} implies $\gamma=0$, i.e. $Q=q$. Suppose to the contrary that $\gamma\neq0$ and choose $\eps$ satisfying
\begin{equation}\label{196}
0<\eps<\min(\tf1p,\tf1{p_1}-\tf1p,\abs{\gamma}).
\end{equation}

\textit{Case $\gamma<0$.} Put
$$f(t)=t^{\eps-\f1p}\chi_{(0,1)}(t),\quad t\in(0,\infty).$$
Then $f\in\Dupa{0,\infty}$ and, using \eqref{196}, we obtain
\begin{align*}
\LHS{195}
&\gs\norm{t^{\gamma+\f1p-\f1{p_1}-\f1R}\la_*(t)\intt{0}{t}{u^{\eps+\f1{p_1}-\f1p-1}}{u}}_{R,(0,1)}\es\norm{t^{\gamma+\eps-\f1R}\la_*(t)}_{R,(0,1)}=\infty,
\end{align*}
while
$$\RHS{195}=\norm{t^{\eps-\f1s}b_*(t)}_{s,(0,1)}\ls1,$$
which gives the contradiction.

\textit{Case $\gamma>0$.} Now put
$$f(t)=\chi_{(0,1)}(t)+t^{-\eps-\f1p}\chi_{[1,\infty)}(t),\quad t\in(0,\infty).$$
Then $f\in\Dupa{0,\infty}$ and, using \eqref{196}, we obtain
\begin{align*}
\LHS{195}
&\gs\norm{t^{\gamma+\f1p-\f1{p_1}-\f1R}\la_*(t)\intt{1}{t}{u^{-\eps+\f1{p_1}-\f1p-1}}{u}}_{R,(1,\infty)}\es\norm{t^{\gamma-\eps-\f1R}\la_*(t)}_{R,(1,\infty)}=\infty,
\end{align*}
while
$$\RHS{195}\es\norm{t^{\f1p-\f1s}b_*(t)}_{s,(0,1)}+\norm{t^{-\eps-\f1s}b_*(t)}_{s,(1,\infty)}\ls1,$$
which is the contradiction.

Thus, under the assumption $T\in\LBl(p_1,q_1;m)$, we have proved $Q=q$. When $T\in\LBr(p_2,q_2;m)$, one can proceed analogously.

Now, using Theorem~\ref{Int} on \eqref{193} with $Q=q$, we obtain $N(s,R,\la_*,b_*;0,\infty)<\infty,$ which, by Lemma~\ref{Temb} ($q<\infty$), implies that $L_{q,s;b}\embl L_{q,R;\la},$ and the proof is complete.
\end{proof}

In the limiting cases one can prove the optimality in the sense mentioned above only in some special cases. This is caused by the fact that, in general, the optimal target or source spaces lie outside the scale of LK spaces (cf. \cite[Section 5]{GOT}). However, we shall mention at least some partial (sharp) results in this direction. Similar results for the sharp embeddings of Bessel-potential-type spaces into LK spaces appeared in \cite{GNO}, for example. For brevity, we shall state the following results only in the case where $\mu_1(R_1)=\mu_2(R_2)=1$. The next theorem describes the limiting case $\theta=0$.

\begin{theorem}\label{OptL}
Let $T\in\JW(p_1,q_1;p_2,q_2)\cap\LBl(p_1,q_1;m)$ be a quasilinear operator.
\begin{itemize}
\item[{\rm(i)}] Let $1<r\leq s\leq\infty$ and suppose $a\in\SV(0,1)$ is such that
\begin{equation}\label{156}
\intt{0}{1}{t^{-1}a(t)^{-r'}}{t}<\infty.
\end{equation}
Define
\begin{equation}\label{eq40}
\beta(t)=a(t^m)^{-\f{r'}{s}}\left(\intt{0}{t^m}{u^{-1}a(u)^{-r'}}{u}\right)^{-\f1{r'}-\f1s},\quad t\in(0,1).
\end{equation}
Then $\beta\in\SV(0,1)$ and
\begin{equation}\label{153}
T:L_{p_1,r;a}\lra L_{q_1,s;\beta}.
\end{equation}

Moreover, if $\lambda\in SV(0,1)$ is such that
\begin{equation}\label{154}
T:L_{p_1,r;a}\lra L_{q_1,s;\lambda}
\end{equation}
and the limit
\begin{equation}\label{158}
\lim_{x\to0_+}\f{\lambda_*(x)}{\beta_*(x)}
\end{equation}
exists when $s<\infty$, then 
\begin{equation}\label{159}
L_{q_1,s;\beta}\embl L_{q_1,s;\lambda}.
\end{equation}

\item[{\rm(ii)}] Let $1\leq r\leq s<\infty$ and suppose $b\in\SV(0,1)$ is such that
$$\intt{0}{1}{t^{-1}b_*(t)^s}{t}=\infty.$$
Define
\begin{equation}\label{eq45}
\alpha(t)=b_*(t)^{-\f{s}{r'}}\left(1+\intt{t}{1}{u^{-1}b_*(u)^s}{u}\right)^{\f1{r'}+\f1s},\quad t\in(0,1).
\end{equation}
Then $\alpha\in \SV(0,1)$ and
\begin{equation*}
T:L_{p_1,r;\alpha}\lra L_{q_1,s;b}.
\end{equation*}

Moreover, if $\lambda\in \SV(0,1)$ is such that
$$T:L_{p_1,r;\lambda}\lra L_{q_1,s;b}$$
and the limit
$$\lim_{x\to0_+}\f{\alpha(x)}{\lambda(x)}$$
exists when $r>1$, then 
$$L_{p_1,r;\lambda}\embl L_{p_1,r;\alpha}.$$
\end{itemize}
\end{theorem}
\begin{proof}
We prove part (i) here; the proof of part (ii) is analogous.

\textit{Case $s<\infty$.} By Lemma~\ref{L} (ii) ($r>1$), the function $\beta$ defined by \eqref{eq40} satisfies
\begin{equation}\label{155}
\norm{t^{-\f1s}\beta_*(t)}_{s,(x,1)}\es\norm{t^{-\f1{r'}}a(t)^{-1}}_{r',(0,x)}^{-1}\quad\forall x\in(0,\tf12).
\end{equation}
Therefore, Theorem~\ref{IntL} (and Remark~\ref{Rem}~(i)) yields \eqref{153}. Moreover, Theorem~\ref{IntL} and \eqref{154} implies that
$$\norm{t^{-\f1s}\lambda_*(t)}_{s,(x,1)}\norm{t^{-\f1{r'}}a(t)^{-1}}_{r',(0,x)}\ls1\quad\forall x\in(0,\tf12).$$
Together with \eqref{155}, this gives
\begin{equation}\label{157}
\f{\intt{x}{1}{t^{-1}\lambda_*(t)^s}{t}}{\intt{x}{1}{t^{-1}\beta_*(t)^s}{t}}\ls1\quad\forall x\in(0,\tf12).
\end{equation}
Since the denominator of $\LHS{157}$ tends to infinity as $x\to0_+$ (see \eqref{155} and \eqref{156}) and we assume that limit \eqref{158} exists, we can apply L'Hospital's rule to $\LHS{157}$ to get
$$1\gs\lim_{x\to0_+}\f{\intt{x}{1}{t^{-1}\lambda_*(t)^s}{t}}{\intt{x}{1}{t^{-1}\beta_*(t)^s}{t}}
=\lim_{x\to0_+}\f{\lambda_*(x)^s}{\beta_*(x)^s}.$$
Thus, by Lemma~\ref{Temb}, we obtain \eqref{159}.

\textit{Case $s=\infty$.} Now \eqref{eq40} reads as $\beta_*(x)=\norm{t^{-\f1{r'}}a(t)^{-1}}_{r',(0,x)}^{-1}$, $x\in(0,1)$. This, \eqref{154} and Theorem~\ref{IntL} yield
$$1\gs\norm{\lambda_*}_{\infty,(x,1)}\norm{t^{-\f1{r'}}a(t)^{-1}}_{r',(0,x)}\gs\lambda_*(x)\beta_*(x)^{-1}$$
for all $x\in(0,1)$, which, by Lemma~\ref{Temb}, implies \eqref{159}.
\end{proof}

It is obvious that the requirement about the existence of the limit in Theorem~\ref{OptL} may be dropped in many situations. For example, this assumption is redundant if $a$ and $b$ are products of composite logarithmic functions.

Next we consider the limiting case $\theta=1$ that is analogous to the previous theorem, except for the case $p_2=\infty$ (thus we shall prove only this case). In order to keep the presentation brief, let us make a convention that if we say that some result is sharp, then we mean it in the sense of the previous theorem (assuming the existence of the corresponding limits when needed).

\begin{theorem}\label{OptR}
Let $T\in\JW(p_1,q_1;p_2,q_2)\cap\LBr(p_2,q_2;m)$ be a quasilinear operator.
\begin{itemize}
\item[{\rm(i)}] Let $p_2<\infty$, $1<r\leq s\leq\infty$, $a\in\SV(0,1)$ and suppose that $\intt{0}{1}{t^{-1}a(t)^{-r'}}{t}=\infty$. Then
$$T:L_{p_2,r;a}\lra L_{q_2,s;\beta},$$
where
\begin{equation}\label{00}
\beta(t)=a(t^m)^{-\f{r'}{s}}\left(1+\intt{t^m}{1}{u^{-1}a(u)^{-r'}}{u}\right)^{-\f1{r'}-\f1s},\quad t\in(0,1),
\end{equation}
is a sharp result.

\item[{\rm(ii)}] Let $p_2<\infty$, $1\leq r\leq s<\infty$, $b\in\SV(0,1)$ and suppose that $\intt{0}{1}{t^{-1}b_*(t)^s}{t}<\infty$. Then
$$T:L_{p_2,r;\alpha}\lra L_{q_2,s;b},$$
where
$$\alpha(t)=b_*(t)^{-\f{s}{r'}}\left(\intt{0}{t}{u^{-1}b_*(u)^{s}}{u}\right)^{\f1{r'}+\f1s},\quad t\in(0,1),$$
is a sharp result.

\item[{\rm(iii)}] Suppose that \eqref{210} holds. The assertions in {\rm(i)} and {\rm(ii)} remain true if $p_2=\infty$, provided that $r=s=\infty$ and $r=s=1$, respectively.
\end{itemize}
\end{theorem}
\begin{proof}\mbox{}

\textit{Case $r=s=\infty$.} By Theorem~\ref{IntRR} and \eqref{210}, the result $T:L_{\infty,\infty;a}\lra L_{q_2,\infty;\lambda}$ implies
$$\norm{\lambda_*}_{\infty,(0,x)}\norm{t^{-1}a(t)^{-1}}_{1,(x,1)}\ls1\quad\forall x\in(0,1),$$
and then we can argue similarly as in the case $s=\infty$ of the proof of Lemma~\ref{OptL} (i).

\textit{Case $r=s=1$.} In this case we can write
\begin{align}\label{1240}
\norm{t^{-\f1s}b_*(t)\log\tf xt}_{s,(0,x)}
=\intt{0}{x}{t^{-1}b_*(t)\intt{t}{x}{u^{-1}}{u}}{t}
&=\intt{0}{x}{u^{-1}\intt{0}{u}{t^{-1}b_*(t)}{t}}{u}\nonumber\\
&=\norm{t^{-1}\alpha(t)}_{1,(0,x)}
\end{align}
for all $x\in(0,1)$, hence, by Theorem~\ref{IntRR}, the result
$$L_{\infty,1;\alpha}\lra L_{\infty,1;b}$$
holds. Now if $\la\in\SV(0,1)$ is such that $L_{\infty,1;\la}\lra L_{\infty,1;b}$, then Theorem~\ref{IntRR} and \eqref{1240} imply
$$1\gs\norm{t^{-\f1s}b_*(t)\log\tf xt}_{s,(0,x)}\norm{t^{-1}\la(t)}_{1,(0,x)}^{-1}
=\norm{t^{-1}\alpha(t)}_{1,(0,x)}\norm{t^{-1}\la(t)}_{1,(0,x)}^{-1},$$
for all $x\in(0,1)$, therefore, by Lemma~\ref{Temb}, we get $L_{\infty,1;\la}\embl L_{\infty,1;\alpha}$.
\end{proof}

Analogous assertions can be formulated on the interval $(0,\infty)$ (i.e. if $\mu_1(R_1)=\mu_2(R_2)=\infty$) for the sums and intersections of the LK spaces. However, since there are many possible configurations to cover and the formulas connecting the s.v. functions remain essentially the same, we shall skip this. When needed, we can extract the sharp results directly from the conditions of our interpolation theorems by assuming that the functions appearing under the supremum in $N$, $L$, $R$, $R_\infty$ are equivalent to $1$ and by using Lemma~\ref{L}.

\begin{remark}
There are situations in which the sharp results of Theorem~\ref{OptR} are also optimal. For example, suppose that the assumptions of Theorem~\ref{OptR}~(i) are satisfied and let $b_s=\beta$, where $\beta$ is defined by \eqref{00}. Furthermore, suppose that $q_2=\infty$. Then, from Lemmas~\ref{L} and \ref{Temb}, we deduce that $L_{\infty,s;b_s}\embl L_{\infty,S,b_S}$ whenever $s\leq S\leq\infty$. Therefore, in this situation, the target space $L_{\infty,r,b_r}$ is optimal (cf. \cite[Remark~3.2.~(iii)]{GNO}).
\end{remark}

\section{Applications of the results}\label{S7}

Our main results can be applied to many familiar operators - we will now give several examples. We shall mention the sharp (or optimal) results only. 

In the following, the symbol $\Meas(\Omega)$ stands for the set of all Lebesgue measurable functions on $\Omega\subseteq\R^n$, $n\in\N$, and $|Q|$ for the Lebesgue measure of the set $Q\subseteq\R^n$.

In the next lemma we recall the definitions of several familiar operators and corresponding well known estimates that are required by our interpolation theorems.

\begin{lemma}\label{Oper}
\begin{itemize}
\item[{\rm(i)}] Let $|\Omega|=1$. \textsc{The Hardy-Littlewood maximal operator} $M_{\Omega}$ is defined for a locally integrable function $f\in\Meas(\Omega)$ by
$$M_{\Omega}f(x)=\sup_{Q\ni x}\f1{|Q|}\int\limits_{Q\cap\Omega}|f|,\quad x\in\Omega,$$
where the supremum extends over all cubes containing $x$ which have sides parallel to coordinate axes. The operator $M_{\Omega}$ satisfies
\begin{equation}\label{1100}
(M_{\Omega}f)^*(t)\es\f1t\intt{0}{t}{f^*(u)}{u}\quad\forall \text{\rm loc. int. } f\in\Meas(\Omega)\quad\forall t\in(0,1),
\end{equation}
and thus $M_{\Omega}\in\JW(1,1;\infty,\infty)\cap\LBl(1,1;1)$.
\item[{\rm(ii)}] \textsc{The conjugate function}, defined for a locally integrable $2\pi$-periodic function $f\in\Meas(\R)$ by
$$\CC f(x)=\f1{\pi}\lim_{\eps\to0_+}\int\limits_{\eps<|t|\leq\pi}\!\!\!\!(2\cot\tf{t}2)f(x-t)\,\mathrm{d}t,\quad x\in\R,$$
satisfies $\CC\in\JW(1,1;\infty,\infty)\cap\CC\in\LBl(1,1;1)\cap\LBr(\infty,\infty;1)$.
\item[{\rm(iii)}] \textsc{The Riesz potential} $I_{\gamma}$, $0<\gamma<n$, defined for a locally integrable function $f\in\Meas(\R^n)$ by
$$I_{\gamma}f(x)=c(n,\gamma)\intt{\R^n}{}{|t|^{\gamma-n}f(x-t)}{t},\quad x\in\R^n,$$
satisfies $I_{\gamma}\in\JW(1,\tf{n}{n-\gamma};\tf{n}{\gamma},\infty)\cap\LBl(1,\tf{n}{n-\gamma};1)\cap\LBr(\tf{n}{\gamma},\infty;1)$.
\item[{\rm(iv)}] \textsc{The Hilbert transform}, defined for every function $f\in\Meas(\R)$ such that $f\in L_1+L_{\infty,1}$ by
\begin{equation*}
Hf(x)=\f1{\pi}\lim_{\eps\to0_+}\int\limits_{\eps<|t|}t^{-1}f(x-t)\,\mathrm{d}t,\quad x\in\R,
\end{equation*}
satisfies $H\in\JW(1,1;\infty,\infty)\cap\LBl(1,1;1)\cap\LBr(\infty,\infty;1)$.
\item[{\rm(v)}] \textsc{The Riesz transforms} $R_i$, $1\leq i\leq n$, defined for all functions $f\in\Meas(\R^n)$ such that $f\in L_1+L_{\infty,1}$ by
$$R_if(x)=c(n)\lim_{\eps\to0_+}\int\limits_{\eps<|t|}{\f{t_i}{|t|^{n+1}}f(x-t)}\,\mathrm{d}t,\quad x\in\R^n,$$
satisfy $R_i\in\JW(1,1;\infty,\infty)\cap\LBl(1,1;1)\cap\LBr(\infty,\infty;1)$.
\end{itemize}
\end{lemma}
\begin{proof}\mbox{}

\textit{Case} (i). See \cite[Chapter 3, Theorem 3.8]{BS}.

\textit{Cases} (ii), (iv). That $\CC,H\in\JW(1,1;\infty,\infty)$ follows from \cite[Chapter 3, Theorem~6.8]{BS} and \cite[Chapter~3, Theorem~4.8.]{BS}. The corresponding lower bound for $H$ is proved in \cite[Proposition~4.10.]{BS} and this proof works as well for $\CC$ (cf. also \cite[Theorem~10.2.~(ii)]{EOP}).

\textit{Case} (iii). See \cite[p.150]{Saw} and references there.

\textit{Case} (v). The Riesz transforms satisfy essentially the same rearrangement inequality as $H$ does (cf. \cite[p.150]{Saw}).
\end{proof}

The following theorem concerns the boundedness of operators $M_{\Omega}$ and $\CC$, which are acting between function spaces over finite measure spaces.

\begin{theorem}\label{TM}
Let $T$ be $M_{\Omega}$ or $\CC$. Then
\begin{align}
T&: & L_{p,s;b} &\lra L_{p,s;b},\quad1<p<\infty, & 1&\leq s\leq\infty,\quad b\in\SV(0,1);\label{M1}\\
T&: & L_{1,1;b}&\lra L_{1,\infty;b},  &b&\in\SV(0,1)\cap\Dupa{0,1};\label{MCL}\\
T&: & L_{1,r;\f1{r'},\f1{r'},\f1{r'}+\alpha}&\lra L_{1,s;-\f1s,-\f1s,-\f1s+\alpha}, & 1&\leq r\leq s\leq\infty,\quad\alpha>0;\label{M5}\\
M_{\Omega}&: & L_{\infty}&\lra L_{\infty};& &\label{MN}\\
\CC&: & L_{\infty} &\lra L_{\infty,\infty;\el^{-1}};& &\label{C1}\\
\CC&: & L_{\infty,1;-1,-1,-1-\alpha}&\lra L_{\infty,1;0,0,-\alpha},& \alpha&>0;\label{C2}\\
\CC&: & L_{\infty,\infty;\,\exp(-\sqrt{\el})}&\lra L_{\infty,\infty;\,\exp(-\sqrt{\el})/\sqrt{\el}}.& &\label{C3}
\end{align}
\end{theorem}
\begin{proof}
Result \eqref{M1} is a consequence of Theorem~\ref{Opt}.

Result \eqref{MCL} follows easily from Theorem~\ref{IntL} (and Remark~\ref{Rem}~(i)), since, in this case, one has
$$L(1,\infty,a,b;0,1)=\norm{b}_{\infty,(x,1)}\norm{a^{-1}}_{\infty,(0,x)}.$$

Result \eqref{M5} follow from Theorem~\ref{OptL}. Indeed, observe that $\beta$ and $\alpha$ from \eqref{eq40} and \eqref{eq45} now take the form
\begin{align*}
\beta(t)&=\el(t)^{-\f1s}\eld(t)^{-\f1s}\elt(t)^{-\f1s-\alpha\f{r'}{s}}\left(\intt{0}{t}{u^{-1}\el(u)^{-1}\eld(u)^{-1}\elt(u)^{-1-\alpha r'}}{u}\right)^{-\f1{r'}-\f1s}\\
&\es\el(t)^{-\f1s}\eld(t)^{-\f1s}\elt(t)^{-\f1s-\alpha\f{r'}{s}}\elt(t)^{\alpha+\alpha\f{r'}{s}}=\el(t)^{-\f1s}\eld(t)^{-\f1s}\elt(t)^{-\f1s+\alpha}
\end{align*}
and
\begin{align*}
\alpha(t)&=\el(t)^{\f1{r'}}\eld(t)^{\f1{r'}}\elt(t)^{\f1{r'}-\alpha\f{s}{r'}}\left(1+\intt{t}{1}{u^{-1}\el(u)^{-1}\eld(u)^{-1}\elt(u)^{-1+\alpha s}}{u}\right)^{\f1{r'}+\f1s}\\
&\es\el(t)^{\f1{r'}}\eld(t)^{\f1{r'}}\elt(t)^{\f1{r'}-\alpha\f{s}{r'}}\elt(t)^{\alpha\f{s}{r'}+\alpha}=\el(t)^{\f1{r'}}\eld(t)^{\f1{r'}}\elt(t)^{\f1{r'}+\alpha}
\end{align*}
for all $t\in(0,1)$.

Result \eqref{MN} follows either from the definition of $M_{\Omega}$, or from \eqref{1100} and Lemma~\ref{TLH} with $\mu=\kappa=1$, $r=s=\infty$, $a=b$. 

Results \eqref{C1} and \eqref{C3} follow from Theorem~\ref{OptR}~(iii) with $r=s=\infty$. Indeed, this is obvious for \eqref{C1} and for \eqref{C3} we set $a(t)=\exp\sqrt{-\el(t)}$, $t\in(0,1)$, use the substitution $\tau=\sqrt{\el(u)}$ and integration by parts to get
\begin{align*}
\intt{t}{1}{u^{-1}a(u)^{-r'}}{u}
=2\intt{1}{\sqrt{\el(t)}}{\tau\exp(\tau)}{\tau}
\es\sqrt{\el(t)}\exp\sqrt{\el(t)}\quad\forall t\in(0,\tf12).
\end{align*}

Result \eqref{C2} is a consequence of Theorem~\ref{OptR}~(iii) with $r=s=1$.
\end{proof}

\begin{remark}\label{RInt}
It is obvious from the proof above that results \eqref{M5} and \eqref{C2} hold analogously for arbitrary tier of logarithms.
\end{remark}

Result \eqref{M5} together with Remark~\ref{RInt} yield many particular results (in fact, also the result \eqref{MCL} with $b\equiv1$ can be seen as the limiting case $\alpha\to0$ of \eqref{M5} with $r=1$, $s=\infty$). Some of these are stated in the following corollary.

\begin{cor}
Let $T$ be $M_{\Omega}$ or $\CC$. Then
\begin{align*}
T&: & L(\log L)&\lra L_{1};& &\\
T&: & L(\log\log\log L)&\lra L(\log L)^{-1}(\log\log L)^{-1};& &\\
T&: & L_{1,\infty;1,1,1+\alpha}&\lra L_{1,1;0,0,\alpha},&\quad\alpha&>0.
\end{align*}
\end{cor}

The result \eqref{M1} for $M_{\Omega}$ with $b\equiv1$ is a well known result of Hardy, Littlewood. The non-limiting case for operator $\CC$ was resolved by F. Riesz. The limiting cases with single logarithm for operator $\CC$ are due to Zygmund. Analogous results for GLZ spaces with second tier of logarithms were proven in \cite{EOP}. The results \eqref{M5}, \eqref{C2}, \eqref{C3} (and their versions for higher tiers of logarithms) are new. The spaces in \eqref{C3} are not GLZ spaces.

Now we shall present some results for operators $I_{\gamma}$, $H$ and $R_i$, acting between function spaces over $\R$ or $\R^n$. We start with $I_{\gamma}$, since its behaviour near the right endpoint is easier to describe than for the other two operators ($p_2=\f{n}{\gamma}<\infty$).

\begin{theorem}\label{TRP}
Let $\g=\f{n}{n-\gamma}$. The operator $I_{\gamma}$ satisfies
\begin{align}
I_{\gamma}&: &L_{p,s;b} &\lra  L_{q,s;b},& 1&<p<\tf{n}{n-\gamma},\quad\tf1p=\tf1q+\tf{\gamma}n,\quad b\in\SV(0,\infty);\label{I1}\\
I_{\gamma}&: &L_1+L_{\f{n}{\gamma},1}&\lra L_{\f{n}{n-\gamma},\infty}+L_{\infty};& &\label{I2}\\
I_{\gamma}&: &L_1\cap L_{\f{n}{\gamma},1}&\lra L_{\f{n}{n-\gamma},\infty}\cap L_{\infty};& &\label{I2b}
\end{align}
\vspace{-0.75cm}
\begin{align}\label{I3}
I_{\gamma}&: &L_{1,r_1;\f1{r_1'},\f1{r_1'}+\alpha}+L_{\f{n}{\gamma},r_2;\f1{r_2'},\f1{r_2'}+\beta}  &\lra L_{\f{n}{n-\gamma},s_1;-\f1{s_1},-\f1{s_1}+\alpha}+L_{\infty,s_2;-\f1{s_2},-\f1{s_2}+\beta},& &\mkern45mu \nonumber\\
& & 1\leq r_1\leq s_1&\leq\infty,\quad 1\leq r_2\leq s_2\leq\infty,\quad \alpha,\beta>0;& &
\end{align}
\vspace{-0.7cm}
\begin{align}\label{I4}
I_{\gamma}&: &L_{1,r_1;\f1{r_1'},\f1{r_1'}-\alpha}\cap L_{\f{n}{\gamma},r_2;\f1{r_2'},\f1{r_2'}-\beta}  &\lra L_{\f{n}{n-\gamma},s_1;-\f1{s_1},-\f1{s_1}-\alpha}\cap L_{\infty,s_2;-\f1{s_2},-\f1{s_2}-\beta},& & \mkern45mu\nonumber\\
& & 1\leq r_1\leq s_1&\leq\infty,\quad 1\leq r_2\leq s_2\leq\infty,\quad \alpha,\beta>0;& &
\end{align}
\end{theorem}
\begin{proof}
The non-limiting case \eqref{I1} follows from Theorem~\ref{Opt}. Result \eqref{I3} follows from Theorem~\ref{IntLim}~(i). Indeed, we know from this theorem that \eqref{I3} holds if and only if 
\begin{align*}
\infty&>L(r_1,s_1,a,b;0,1)+R(r_2,s_2,a,b;1,\infty)\\
&=\norm{t^{-\f1{s_1}}\el(t)^{-\f1{s_1}}\eld(t)^{-\f1{s_1}+\alpha}}_{s_1,(x,1)}\norm{t^{-\f1{r_1'}}\el(t)^{-\f1{r_1'}}\eld(t)^{-\f1{r_1'}-\alpha}}_{r_1',(0,x)}\\
&\qquad+\norm{t^{-\f1{s_2}}\el(t)^{-\f1{s_2}}\eld(t)^{-\f1{s_2}+\beta}}_{s_2,(1,x)}\norm{t^{-\f1{r_2'}}\el(t)^{-\f1{r_2'}}\eld(t)^{-\f1{r_2'}-\beta}}_{r_2',(x,\infty)}\\
&\es\eld(x)^{\alpha}\eld(x)^{-\alpha}+\eld(x)^{\beta}\eld(x)^{-\beta}\es1
\end{align*}
for all $x\in(0,\infty)$. The proof of \eqref{I4} is analogous (use Theorem~\ref{IntLim}~(ii)). Results \eqref{I2}, \eqref{I2b} can be seen as the limiting cases $\alpha,\beta\to0$ of \eqref{I3}, \eqref{I4}, respectively, and their proof is similar to the proof of \eqref{MCL}.
\end{proof}

Similarly as for Theorem~\ref{TM}, results \eqref{I3} and \eqref{I4} (and their versions for other tiers of logarithms) yield many particular results. These generalize some of those given in \cite{EOP} to measure spaces with $\mu_1(R_1)=\mu_2(R_2)=\infty$.

\begin{cor}
Let $\g=\f{n}{n-\gamma}$. The operator $I_{\gamma}$ satisfies
\begin{align*}
I_{\gamma}&: &L_{1,1;1,0}+L_{\g ',\g';1,0} &\lra  L_{\g,\g;\f1{\g'},0}+L_{\infty,\infty;\f1{\g'},0};\\
I_{\gamma}&: &L_{1,1;\f1{\g},0}+L_{\g',\g';1,0}&\lra L_{\g}+L_{\infty,\g'};\\
I_{\gamma}&: &L_{1,1;-1,0}\cap L_{\g',\g';\f1{\g},0}&\lra L_{\g,\infty;-1,0}\cap L_{\infty,\infty;0,-\f1{\g}};\\
I_{\gamma}&: &L_{1,1;0,\f1{\g}}+L_{\g',\g';\f1{\g},1}&\lra L_{\g,\g;-\f1{\g},0}+L_{\infty,\infty;0,\f1{\g'}};\\
I_{\gamma}&: &L_{1,1;0,-\f1{\g'}}\cap L_{\g',\g';\f1{\g},0}&\lra L_{\g,\g;-\f1{\g},-1}\cap L_{\infty,\infty;0,-\f1{\g}}. 
\end{align*}
\end{cor}

We shall conclude the paper with the application to the operators $H$ and $R_i$.

\begin{theorem}\label{THil}
Let $T$ be one of the operators $H$, $R_i$. Then
\begin{align*}
T&: &L_{p,s;b}&\lra L_{p,s;b} & 1&<p<\infty,\quad b\in\SV(0,\infty);\\
T&: &L_1+L_{\infty,1}&\lra L_{1,\infty}+L_{\infty}; & &\\
T&: &L_{1,1;1,0}+L_{\infty,1}&\lra L_{1}+L_{\infty}; & &\\
T&: &L_{1,1;1,0}+L_{\infty,1;1,0}&\lra L_{1}+L_{\infty,1};& &\\
T&: &L_{1,1;0,1}+L_{\infty,1;0,1}&\lra L_{1,1;-1,0}+L_{\infty,1;-1,0};& &\\
T&: &L_{1,1;1,0}\cap L_{\infty,1;0,-\alpha}&\lra L_{1}\cap L_{\infty,1;-1,-1-\alpha}& \alpha&>0;\\
T&: &L_{1,1;1,0}+L_{\infty,\infty;1,1+\alpha}&\lra L_{1}+L_{\infty,\infty;0,\alpha}& \alpha&>0;\\
T&: &L_{1}\cap L_{\infty}&\lra L_{1,\infty}\cap L_{\infty,\infty;-1,0};& &\\
T&: &L_{1}\cap L_{\infty,\infty;1,0}&\lra L_{1,\infty}\cap L_{\infty,\infty;0,-1}.& &
\end{align*}
\begin{proof}
The proof can be done using similar ideas as in the proof of Theorem~\ref{TRP}. Instead of Theorem~\ref{IntLim}, we use Theorem~\ref{IntHil}.
\end{proof}
\end{theorem}
The results contained in Theorem~\ref{THil} again generalize some of the results of \cite{EOP} and extend those of \cite{BR}.

\end{document}